\documentclass[10pt,reqno]{amsart}
\usepackage{latexsym}
\usepackage{leftidx}
\usepackage{amsmath,amsthm, color, pdfpages}
\usepackage{enumerate}
\usepackage{amssymb}
\usepackage{upgreek}
\usepackage[margin=1.1in,dvips]{geometry}

\def\f12{\frac 1 2}

\def\a{\alpha}
\def\b{\beta}
\def\ga{\gamma}

\def\ep{\epsilon}

\def\La{\Lambda}
\def\si{\sigma}

\def\om{\omega}
\def\Om{\Omega}

\def\H{\mathcal{H}} 
\def\S{\mathcal{S}} 

\def\Lb{\underline{L}}

\def\Hb{\underline{\H}}

\def\pa{\partial}
\def\les{\lesssim}
\def\B{\mathcal{B}}
\def\cL{{\mathcal L}}
\def\N{\mathcal{N}}
\def\cJ{\mathcal{J}}
\def\cE{\mathcal{E}} 
\def\cM{\mathcal{M}} 
\def\cD{\mathcal{D}} 
\def\f12{\frac 1 2}
\def\div{\text{div}}

\def\Hy{\mathbb{H}}
\newcommand{\vol}{\textnormal{vol}}

\newcommand{\nabb}{\mbox{$\nabla \mkern-13mu /$\,}}

\newtheorem{Thm}{Theorem}[section]

\newtheorem{Prop}{Proposition}[section]

\newtheorem{Lem}{Lemma}[section]

\newtheorem{Cor}{Corollary}[section]

\newtheorem{Remark}{Remark}[section]
\theoremstyle{definition}

\begin{document}

\title{Pointwise decay for semilinear wave equations in $\mathbb{R}^{1+3}$}
\date{}
\author{Shiwu Yang}
\begin{abstract}
In this paper, we use Dafermos-Rodnianski's new vector field method to study the asymptotic pointwise decay properties for solutions of energy subcritical defocusing semilinear wave equations in $\mathbb{R}^{1+3}$. We prove that the solution decays as quickly as linear waves for $p>\frac{1+\sqrt{17}}{2}$, covering part of the sub-conformal case, while for the range $2<p\leq \frac{1+\sqrt{17}}{2}$, the solution still decays with rate at least $t^{-\frac{1}{3}}$. As a consequence, the solution scatters in energy space when $p>2.3542$. We also show that the solution is uniformly bounded when $p>\frac{3}{2}$.
\end{abstract}
\maketitle
\section{Introduction}
This paper is devoted to studying the global dynamics of solutions to the energy subcritical defocusing semilinear wave equation
\begin{equation}
  \label{eq:NLW:semi:3d}
  \Box\phi=|\phi|^{p-1}\phi,\quad \phi(0, x)=\phi_0(x),\quad \pa_t\phi(0, x)=\phi_1(x)
\end{equation}
in $\mathbb{R}^{1+3}$ with $1<p<5$.

The existence of global classical solutions has early been obtained by J\"{o}rgens in \cite{Jorgens61:energysub:NLW:lowerd}, with various kinds of extensions in \cite{Brenner79:globalregularity:NW}, \cite{Brenner81:globalregularity:d9}, \cite{velo85:global:sol:NLW}, \cite{Velo89:globalsolution:NLW}, \cite{Pecher76:NLW:global}, \cite{segal63:semigroup}, \cite{Vonwahl75:NW} and references therein.
 Strauss in \cite{Strauss:NLW:decay} investigated the asymptotics of the global solutions for the super-conformal case $3\leq p<5$. He showed that the solution scatters to linear wave in $H^1$ with compactly supported initial data (see detailed discussion regarding the scattering results in the author's companion paper \cite{yang:scattering:NLW} and references therein). This scattering result relied on the pointwise decay estimate $t^{\ep-1}$ with $\epsilon >0$, which has later been improved to $t^{-1}$ for the strictly super-conformal case $3<p<5$ and to $t^{-1}\ln t$ for the conformal case $p=3$ by Wahl in \cite{vonWahl72:decay:NLW:super}. The starting point of these asymptotic decay estimates is the conservation of approximate conformal energy derived by using the conformal Killing vector field as multiplier. The super-conformal structure of the equation plays the role that the corresponding conformal energy (including the potential energy contributed by the nonlinearity) is controlled by the initial data as the spacetime integral arising from the nonlinearity is nonnegative, which in particular implies the uniform bound of the $L^2$ norm of the solution and the time decay of the potential energy. These a priori bounds lead to the time decay of the solutions.
Another geometric point of view to see this conformal structure is the method of conformal compactification (see \cite{ChristodoulouYangM}, \cite{ChDNull}), based on which together with the representation formula for linear wave equation,
Bieli-Szpak obtained shaper decay estimates for the solution with compactly supported initial data in \cite{Bieli:3DNLW}.

To go beyond the super-conformal case, Pecher in \cite{Pecher82:decay:3d} observed that the potential energy still decays in time but with a weaker decay rate. For this range of power $1<p<3$, the spacetime integral arising from the nonlinearity mentioned above changes sign to be negative. This term could be controlled by using Gronwall's inequality with the price that the conformal energy grows in time with a rate linearly depending on the coefficient. Since the conformal energy contains weights $t^2$, the potential energy still decays in time when $p$ is not too small (sufficiently close to $3$ so that the coefficient is small). This weaker energy decay estimate is sufficiently strong to show the pointwise decay estimates for the solutions when $p>\frac{1+\sqrt{13}}{2}$. As a consequence, the solution scatters in energy space for $p>2.7005$.

The above discussion already indicates that small power $p$ makes the analysis more involved. Since we are expecting the solution decays in time, the larger power $p$ leads to faster decay of the nonlinearity. 
 The aim of this paper is to derive quantitative pointwise decay estimates for the solutions with small power $p$ by using the vector field method originally introduced by Dafermos and Rodnianski in \cite{newapp}.

To state our main results,  for $\gamma \geq 0$, define the weighted energy norm of the initial data
\begin{align*}
  \mathcal{E}_{k,\ga} =\sum\limits_{l\leq k}\int_{\mathbb{R}^3}(1+|x|)^{\ga+2l}(|\nabla^{l+1}\phi_0|^2+|\nabla^l \phi_1|^2)+(1+|x|)^{\ga}|\phi_0|^{p+1}dx.
\end{align*}
Our first result gives an improved lower bound for the power $p$ such that the solution decays as quickly as linear waves. 
\begin{Thm}
\label{thm:main}
 Consider the Cauchy problem to the energy subcritical defocusing semilinear wave equation \eqref{eq:NLW:semi:3d}. For initial data $(\phi_0, \phi_1)$ bounded in $\mathcal{E}_{1, \ga_0} $ for some constant $1<\ga_0<2$, the solution is global in time and  satisfies the following decay estimates:
\begin{itemize}
\item[1]. For the case when
\[
\frac{1+\sqrt{17}}{2}<p<5, \quad \max\{\frac{4}{p-1}-1, 1\}<\ga_0<\min\{p-1, 2\},
\]
then
\begin{equation*}
|\phi(t, x)|\leq C \sqrt{\mathcal{E}_{1, \ga_0} } (1+t+|x|)^{-1}(1+||x|-t|)^{-\frac{\ga_0-1}{2}};
\end{equation*}
\item[2]. If $2<p\leq \frac{1+\sqrt{17}}{2}$ and $1<\ga_0<p-1$, then
\begin{equation*}
|\phi(t, x)|\leq C \sqrt{\mathcal{E}_{1, \ga_0}  }  (1+t+|x|)^{-\frac{3+(p-2)^2}{(p+1)(5-p)}\ga_0}(1+||x|-t|)^{-\frac{\ga_0}{p+1}}
\end{equation*}
for some constant $C$ depending on $\ga_0$, $p$ and the zeroth order weighted energy $\mathcal{E}_{0, \ga_0} $.
\end{itemize}
 \end{Thm}

 As a consequence of the above pointwise decay estimate, we extend Pecher's scattering result to a larger range of $p$. Recall the linear operator $\mathbf{L}(t)$ defined in \cite{yang:scattering:NLW}
 \begin{align*}
   \Box\mathbf{L}(t)(f, g)=0,\quad \mathbf{L}(0)(f, g)=f(x),\quad \pa_t \mathbf{L}(0)(f, g)=g(x).
 \end{align*}
\begin{Cor}
 \label{cor:scattering:3D}
 For $p>p_*$ (defined in the Section 6 and $p_*<2.3542$) and initial data bounded in $\mathcal{E}_{1, p-1}$, the solution $\phi$ of \eqref{eq:NLW:semi:3d} is uniformly bounded in the following mixed spacetime norm
 \begin{align*}
   \|\phi\|_{L_t^p L_x^{2p}}<\infty.
 \end{align*}
 Consequently the solution scatters in energy space, that is, there exists pairs $(\phi_0^{\pm}(x), \phi_1^{\pm}(x))$ such that
 \begin{align*}
      \lim\limits_{t\rightarrow\pm\infty}\|\phi(t, x)-\mathbf{L}(t)(\phi_0^{\pm}(x), \phi_1^{\pm}(x))\|_{\dot{H}_x^1}+\| \pa_t\phi(t,x)-\pa_t \mathbf{L}(t)(\phi_0^{\pm}(x), \phi_1^{\pm}(x))\|_{ L_x^2}=0.
\end{align*}
\end{Cor}
We give several remarks.
\begin{Remark}
 One can also derive the pointwise decay estimates for the derivatives of the solution by assuming the boundedness of the second order weighted energy of the initial data.
\end{Remark}

\begin{Remark}
 The precise decay estimate obtained by Pecher in \cite{Pecher82:decay:3d} is the following
\[
|\phi|\leq C t^{\frac{6+2p-2p^2}{3+p}+\ep},\quad \frac{1+\sqrt{13}}{2}<p\leq 3
\]
with initial data bounded in $\mathcal{E}_{1, 2} $. Theorem \ref{thm:main} improves this decay estimate with weaker assumption on the initial data for a larger range of power $p$.
\end{Remark}

\begin{Remark}
  Note that the solution to the linear wave equation
  \[
  \Box\phi^{lin}=0,\quad \phi^{lin}(0, x)=\phi_0(x),\quad \pa_t\phi^{lin}(0, x)=\phi_1
  \]
  with data $(\phi_0, \phi_1)$ bounded in $\mathcal{E}_{1, \ga_0} $ for some $1<\ga_0\leq 2$ has the following pointwise decay property
  \begin{align*}
    |\phi^{lin}(t, x)|\leq C \sqrt{\mathcal{E}_{1, \ga_0}}(1+t+|x|)^{-1}(1+|t-|x||)^{-\frac{\ga_0-1}{2}}
  \end{align*}
  for some universal constant $C$. Thus when $\frac{1+\sqrt{17}}{2}<p<5$, for arbitrary large data $(\phi_0, \phi_1)$ bounded in $\mathcal{E}_{1, \ga_0} $, the solution to the nonlinear equation \eqref{eq:NLW:semi:3d} decays as quickly as the solution to the linear equation with the same initial data. This pointwise decay property is consistent with the scattering result obtained in the author's companion paper \cite{yang:scattering:NLW}, in which it has been shown that the solution to \eqref{eq:NLW:semi:3d} scatters in critical Sobolev space $\dot{H}^{\frac{3}{2}-\frac{2}{p-1}}$ and the energy space $\dot{H}^{1}$ when $\frac{1+\sqrt{17}}{2}<p<5$.
\end{Remark}
\begin{Remark}
  Our scattering result in energy space applies to power even below the Strauss exponent $p_c=1+\sqrt{2}$, for which small data global solution and scattering hold for the pure power semilinear wave equation with power above $p_c$ (see for example \cite{Pecher88:scattering:sharpp:3D}) while finite time blow up can occur with power below $p_c$ (see John's work in \cite{John79:blowup:NLW:3d}).
\end{Remark}
Our second result concerns the asymptotic behavior of the solution with small power $p\leq 2$.
\begin{Thm}
\label{thm:main:p2}
Consider the defocusing semilinear wave equation \eqref{eq:NLW:semi:3d} with initial data $(\phi_0, \phi_1)$ such that $\mathcal{E}_{1, 0} $ and $\mathcal{E}_{0, p-1} $ are finite. Then for all $\frac{3}{2}<p\leq 2$,
the solution $\phi$ to the equation \eqref{eq:NLW:semi:3d} exists globally  and is uniformly bounded in the following sense
\begin{equation*}
|\phi(t, x)|\leq C(\sqrt{\mathcal{E}_{1, 0} }+\mathcal{E}_{0, p-1}^{\frac{p}{p+1}}),\quad \forall (t, x)\in \mathbb{R}^{1+3}
\end{equation*}
for some constant $C$ depending only on $p$.
\end{Thm}
\begin{Remark}
The proof shows that the solution actually decays in the spatial variable $x$ but not in time $t$. We may note that there is a gap between the case $p>2$ and $p\leq 2$ regarding the time decay of the solution. Our first result Theorem \ref{thm:main} indicates that the solution decays at least $t^{-\frac{1}{3}}$ when $p>2$ while only boundedness of the solution can be obtained in Theorem \ref{thm:main:p2} for $p\leq 2$. We will fill this gap by using a new approach in our future work \cite{YangW:NLW:3d:p1}.
\end{Remark}

We now discuss the main ideas for the proof. As mentioned above, the existing approach (see for example \cite{Velo87:decay:NLW}, \cite{Pecher82:decay:3d}) to study the asymptotic behavior of solutions to \eqref{eq:NLW:semi:3d} relied on the following time decay of the potential energy
\begin{align}
\label{eq:timedecay:3D:Pecher}
  \int_{\mathbb{R}^3}|\phi|^{p+1}dx\leq C (1+t)^{\max\{4-2p, -2\}},\quad 1<p<5,
\end{align}
which is based on the energy estimate
\begin{align*}
  \int_{\mathbb{R}^{3}}t^2 |\phi|^{p+1}(t, x)dx+\int_0^{t}\int_{\mathbb{R}^3}(2p-6)s|\phi|^{p+1}(s, x)dxds\leq C\mathcal{E}_{0, 2}
\end{align*}
 derived by using the conformal Killing vector field $t^2\pa_t+r^2 \pa_r$ ($r=|x|$) as multiplier. Here the constant $C$ depends only on $p$. With this a priori decay estimate for the solution, a type of $L^q$ estimate for linear wave equation (prototype of Strichartz estimate, see for example \cite{Brenner75:Lp:LW}) yields the pointwise decay estimate for the solution. This approach only makes use of the time decay of the solution. However, it is well known that linear waves enjoy improved decay away from the light cone, which can be quantified in terms of the distance $u=t-|x|$ to the light cone. Our improvement comes from thoroughly utilizing such $u$ decay of linear waves.

 The method we used to explore this $u$ decay is the vector field method originally introduced by Dafermos-Rodnianski in \cite{newapp}. The new ingredient is the $r$-weighted energy estimate derived by using the vector field $r^{\ga}(\pa_t+\pa_r)$ as multiplier with $0\leq \ga\leq 2$.
  Applying this to equation \eqref{eq:NLW:semi:3d}, we obtain that
  \begin{align*}
    \iint_{\mathbb{R}^{1+3}}\frac{p-1-\ga}{p+1}r^{\ga-1}|\phi|^{p+1}dxdt\leq C \mathcal{E}_{0, \ga}.
  \end{align*}
  To get a useful estimate for the solution, we require that $0<\ga<p-1$. On the other hand, combined with an integrated local energy decay estimate obtained by using the vector field $f(r)\pa_r$ as multiplier and the classical energy estimate, the energy flux through the outgoing null hypersurface $\H_u$ (constant $u$ hypersurface) decays in $u$, that is,
  \begin{align*}
    \int_{\H_u} |\phi|^{p+1}d\sigma \leq C(1+|u|)^{-\ga}\mathcal{E}_{0, \ga}.
  \end{align*}
Integrating in $u$, we then get a weighted spacetime bound
\begin{align*}
   \iint_{\mathbb{R}^{1+3}} (1+|u|)^{\ga-1-\ep}|\phi|^{p+1}dxdt\leq C\mathcal{E}_{0, \ga},\quad \forall \ep>0
\end{align*}
by assuming that $\ga>1$ (this in particular forces $p>2$). This together with the above $r$-weighted energy estimate leads to the spacetime bound
\begin{align*}
  \iint_{\mathbb{R}^{1+3}}(1+t+|x|)^{\ga-1-\ep}|\phi|^{p+1}dxdt\leq C\mathcal{E}_{0, \ga},
\end{align*}
which is one of the main results in \cite{yang:scattering:NLW} as restated precisely in Proposition \ref{prop:spacetime:bd}. For the subconformal case $p<3$, since $\ga$ can be as large as $p-1$, in terms of time decay, this spacetime bound is stronger than \eqref{eq:timedecay:3D:Pecher} as $p-2>2p-5$. Our improvement on the asymptotic decay properties of the solution heavily relies on this uniform spacetime bound.

To show the pointwise decay estimates of solutions to \eqref{eq:NLW:semi:3d}, we start by obtaining a uniform weighted energy flux bound through the backward light cone $\N^{-}(q)$ emanating from the point $q=(t_0, x_0)$, that is, $t-t_0+|x-x_0|=0$. Consider the vector field
\begin{align*}
  X^{\gamma}=u_+^{\ga}(\pa_t-\pa_r)+v_+^{\ga}(\pa_t+\pa_r),\quad v_+=\sqrt{1+(t+|x|)^2},\quad u_+=\sqrt{1+u^2}.
\end{align*}
The case when $\ga=2$ corresponds to the conformal Killing vector field while the case when $1<\ga<2$ has been widely used (see for examples
\cite{Dafermos17:C0Kerr}, \cite{LindbladMKG}). Applying this vector field as multiplier to the region bounded by the backward light cone $\N^{-}(q)$ and the initial hypersurface, we obtain that
\begin{align}
\label{eq:intro:0}
  \int_{\mathcal{N}^{-}(q)}\big((1+\frac{x\cdot(x-x_0)}{|x||x-x_0|})v_+^{\ga}+u_+^{\ga}\big)|\phi|^{p+1}
d\sigma\leq C \mathcal{E}_{0, \ga_0},\quad \ga<\ga_0,
\end{align}
 for which the previous uniform spacetime bound controls the spacetime error term (see details in Proposition \ref{prop:EF:cone:NW:3d}).
 Once we have this uniform potential energy bound, we apply the representation formula to show the pointwise decay estimates for the solutions. The nonlinear term can be controlled by interpolation between the $L^\infty$ norm of the solution and the above potential energy decay. This idea is inspired by the method in \cite{Moncrief1}, \cite{Moncrief2} for proving the global solutions for Yang-Mills-Higgs equations. 

However, the procedure involves delicate estimates for the integrals of functions $u_+$, $v_+$ on the backward light cone $\mathcal{N}^{-}(q)$. For technical reasons, we discuss two cases: in the exterior region $\{t+2\leq |x|\}$ and in the interior region $\{t+2>|x|\}$ (by symmetry it suffices to consider the solution in the future $t\geq 0$). When $q$ locates in the exterior region, the backward light cone $\mathcal{N}^{-}(q)$ also lies in the exterior region. This helps to estimate the mentioned integrals. 
By using a type of bootstrap argument (or Gronwall's inequality), we then can derive the pointwise decay estimates of the solutions in the exterior region. 

When $q$ lies in the interior region, estimating integrals of functions of $u_+$, $v_+$ on $\mathcal{N}^{-}(q)$ becomes quite complicated. Instead, we use the method of conformal compactification. 
Pick up the hyperboloid $\mathbb{H}$ passing through the 2-sphere $\{t=0, |x|=2\}$. The region enclosed by $\mathbb{H}$ contains the interior region and is conformally equivalent to the compact backward cone $\{\tilde{t}+|\tilde{x}|\leq \frac{5}{6}\}$. The study of the asymptotic behavior of solutions to \eqref{eq:NLW:semi:3d} in the interior region is then reduced to control the growth of solutions to a class of semilinear wave equation on this compact region with initial data determined by the original solutions on the hyperboloid $\mathbb{H}$, which has already been understood. The argument to control the solution on the compact region is similar to that used for studying the solution in the exterior region.

The key point for the case when $p>2$ is that we are allowed to use the vector field $X^{\ga}$ with $\ga>1$, which in particular requires that $p>2$ (recall that the restriction $\ga<p-1$ is used to make the sign of the weighted potential energy be positive). Since the vector field $X^{\gamma}$ with $\gamma<1$ fails to be useful even for linear waves, the above argument does not work for the case when $p\leq 2$. However, inspired by the above argument, it is crucial to obtain some type of weighted potential energy estimate through backward light cones. The key observation now is that instead of using the vector field $X^{\ga}$ as multiplier, we can still use the vector field $r^{\ga}(\partial_t+\partial_r)$ with $0\leq \ga\leq 1$ to obtain that
 \begin{align*}
  \int_{\mathcal{N}^{-}(q)}\big((1+\frac{x\cdot(x-x_0)}{|x||x-x_0|})r^{\ga}+1\big)|\phi|^{p+1}d\sigma
\leq C \mathcal{E}_{0, \ga},\quad 0<\ga<p-1.
\end{align*}
This estimate is sufficient to conclude that the solution is uniformly bounded when $p>\frac{3}{2}$.

The plan of the paper is as follows: In section 2, we define some notations. In section 3, we use vector field method to derive a uniform weighted energy estimate for the solution through backward light cones. Then in section 4, we obtain the pointwise decay estimates for the solutions in the exterior region. In addition, we show necessary decay estimates for the linear evolutions in the interior region. In section 5, we study a class of semilinear wave equation on a compact backward cone. The approach is similar but this section is independent of others. In section 6, we conduct the conformal transformation and apply the results of Section 4 and 5 to conclude the pointwise decay estimates for the solutions in the interior region. The last section is devoted to the proof for the uniform boundedness of the solutions when $p>\frac{3}{2}$.

\section{Preliminaries and notations}
\label{sec:notation}

We use the standard polar local coordinate system $(t, r,
\om)$ of Minkowski space as well as the null coordinates $u=\frac{t-r}{2}$, $v=\frac{t+r}{2}$ with $\om$ parameterizing the unit two-sphere $\mathbb{S}^2$.  
Introduce a null frame $\{L, \Lb, e_1, e_2\}$ with
\[
L=\pa_v=\pa_t+\pa_r,\quad \Lb=\pa_u=\pa_t-\pa_r
\]
and $\{e_1, e_2\}$ an orthonormal basis of the sphere with
constant radius $r$. At any fixed point $(t, x)$, we may choose $e_1$, $e_2$ such that
\begin{equation}
 \label{eq:nullderiv}
 \begin{split}
&\nabla_{e_i}L=r^{-1}e_i,\quad \nabla_{e_i}\Lb=-r^{-1}e_i, \quad \nabla_{e_1}e_2=\nabla_{e_2}e_1=0,\quad \nabla_{e_i}e_i=-r^{-1}\pa_r,
 \end{split}
\end{equation}
where $\nabla$ is the covariant derivative in Minkowski space. The covariant derivative on the sphere with radius $r$ is $\nabb$. Define the functions
\[
u_+=\sqrt{1+u^2},\quad v_+=\sqrt{1+v^2}.
\]
Through out this paper, the exterior region will be referred as $\{(t, x)| u=\frac{t-|x|}{2}\leq -1, t\geq 0\}$ while the interior region will be $\{(t, x)| u\geq -1, t\geq 0\}$.


Define the hyperboloid
\begin{equation}
\label{eq:def4Hyperboloid}
\Hy:=\left\{(t, x)|(t^*)^2-|x|^2=(R^*)^{-1} t^*\right\}, \quad t^*=t+3, \quad R^*=\frac{3}{5},
\end{equation}
which splits into the future part $\Hy^{+}=\Hy\cap \{t\geq 0\}$ and the past part $\Hy^{-}=\Hy\cap \{t<0\}$. We may note here that the interior region defined above lies inside this hyperboloid.

For any $q=(t_0, x_0)\in \mathbb{R}\times \mathbb{R}^3$ and $r>0$, denote $\B_q(r)$ as the 3-dimensional ball at time $t_0$ with radius $r$ centered at point $q$, that is,
\begin{align*}
  \B_q( r)=\{(t,x)| t=t_0, |x-x_0|\leq r\}.
\end{align*}
 The boundary of $\B_q( r)$ is the 2-sphere $\S_q(r)$. On the initial hypersurface $\{t=0\}$, we use $\B_{r_1}^{r_2}$ to denote the annulus $\{t=0, r_1\leq |x|\leq r_2\}$.
Let $\mathcal{N}^{-}(q)$ be the past null cone of the point $q$
, that is,
 \begin{align*}
   \mathcal{N}^{-}(q):=\{(t, x)| |t-t_0|=|x-x_0|,\quad t\geq 0\}.
 \end{align*}
For simplicity we only consider the solution in the future $\{t\geq 0\}$.
 The region enclosed by this cone is the past of the point $q$ and we denote it as $\mathcal{J}^{-}(q)$, that is,
   \begin{align*}
   \mathcal{J}^{-}(q):=\{(t, x)| |x-x_0|\leq |t-t_0|,\quad t\geq 0\}.
 \end{align*}
Additional to the above introduced coordinates system, we may also use  the new coordinates $(\tilde{t}, \tilde{x})$ centered at $q=(t_0, x_0)$, that is, 
\[
\tilde{t}=t-t_0,\quad \tilde{x}=x-x_0,\quad \tilde{r}=|\tilde{x}|,\quad \tilde{\om}=\frac{\tilde{x}}{|\tilde{x}|},\quad \tilde{u}=\f12 (\tilde{t}-\tilde{r}),\quad \tilde{v}=\f12(\tilde{t}+\tilde{r}).
\]
We also have the associated null frame $\{\tilde{L}, \tilde{\Lb}, \tilde{e}_1, \tilde{e}_2\}$ verifying the same relation \eqref{eq:nullderiv}. Under this new coordinates, the past null cone $\mathcal{N}^{-}(q)$ can be characterized by $\{\tilde{v}=0\}\cap\{0\leq t\leq t_0\}$. Through out this paper, the coordinates $(\tilde{t}, \tilde{x})$ are always referred to the translated ones centered at the point  $q=(t_0, x_0)$ unless it is clearly emphasized.

For simplicity, for integrals in this paper, we will omit the volume form unless it is specified. More precisely we will use
\begin{align*}
  \int_{\mathcal{D}}f,  \quad \int_{\mathcal{N}^{-}(q)}f, \quad \int_{\{t=c\}} f
\end{align*}
to be short for
\begin{align*}
  \int_{\mathcal{D}}f dxdt,   \quad \int_{\mathcal{N}^{-}(q)}f 2\tilde{r}^{2}d\tilde{r}d\tilde{\om}, \quad \int_{\{t=c\}} f dx
\end{align*}
respectively.

Finally we make a convention through out this paper to avoid too many constants that $A\les B$ means that there exists a constant $C$, depending possibly on $p$, $\ga_0$ the weighted energy $\mathcal{E}_{0, \ga_0} $ such that $A\leq CB$.

\section{The weighted energy estimates through backward light cones}
In this section, we establish a uniform weighted energy flux bound on any backward light cone in terms of the zeroth order initial energy, based on the spacetime bound for the solution derived in the author's companion paper \cite{yang:scattering:NLW}, from where we recall the following:
\begin{Prop}
  \label{prop:spacetime:bd}
  For all $2<p\leq 5 $, $1<\ga_0<\min\{2,  p-1\}$ and $\ep>0$, the solution $\phi$ of \eqref{eq:NLW:semi:3d} is uniformly bounded in the following sense
\begin{align}
  \label{eq:spacetime:bd}
  \iint_{\mathbb{R}^{1+3}} v_+^{\ga_0-\ep-1}|\phi|^{p+1}dxdt \leq C \mathcal{E}_{0, \ga_0} 
\end{align}
for some constant $C$ relying only on $p$, $\ep$ and $\ga_0$.
\end{Prop}
\begin{proof}
  See the main theorem in \cite{yang:scattering:NLW}.
\end{proof}
Using this spacetime bound, we establish the following uniform weighted energy flux bound.
\begin{Prop}
\label{prop:EF:cone:NW:3d}
Let $q=(t_0, x_0)$ be any point in $\mathbb{R}^{1+3}$ with $t_0\geq 0$. Then for solution $\phi$ of the nonlinear wave equation \eqref{eq:NLW:semi:3d} and for all $1< \ga<\ga_0< \min\{2, p-1\}$, we have the following uniform bound
\begin{equation}
\label{eq:Eflux:ex:EF}
\begin{split}
&\int_{\mathcal{N}^{-}(q)}\big((1+\tau)v_+^{\ga}+u_+^{\ga}\big)|\phi|^{p+1} 
\leq C \mathcal{E}_{0, \ga_0} 
\end{split}
\end{equation}
for some constant $C$ depending only on $p$, $\ga_0$ and $\ga$. Here $\tau=\om\cdot \tilde{\om}$ and the tilde components are measured under the coordinates $(\tilde{t}, \tilde{x})$ centered at the point $q=(t_0, x_0)$.
\end{Prop}
\begin{proof}
Recall the energy momentum tensor for the scalar field $\phi$
\begin{align*}
  T[\phi]_{\mu\nu}=\pa_{\mu}\phi\pa_{\nu}\phi-\f12 m_{\mu\nu}(\pa^\ga \phi \pa_\ga\phi+\frac{2}{p+1} |\phi|^{p+1}),
\end{align*}
where $m_{\mu\nu}$ is the flat Minkowski metric on $\mathbb{R}^{1+3}$. Then we can compute that
\begin{align*}
  \pa^\mu T[\phi]_{\mu\nu}=&(\Box\phi-  |\phi|^{p-1}\phi)\pa_\nu\phi.
\end{align*}
Now for any vector fields $X$, $Y$ and any function $\chi$, define the current
\begin{equation*}
J^{X, Y, \chi}_\mu[\phi]=T[\phi]_{\mu\nu}X^\nu -
\f12\pa_{\mu}\chi \cdot|\phi|^2 + \f12 \chi\pa_{\mu}|\phi|^2+Y_\mu.
\end{equation*}
Then for solution $\phi$ of equation \eqref{eq:NLW:semi:3d}, we have the energy identity
\begin{equation}
\label{eq:energy:id}
\iint_{\mathcal{D}}\pa^\mu  J^{X,Y,\chi}_\mu[\phi] d\vol =\iint_{\mathcal{D}}div(Y)+ T[\phi]^{\mu\nu}\pi^X_{\mu\nu}+
\chi \pa_\mu\phi\pa^\mu\phi -\f12\Box\chi\cdot|\phi|^2 +\chi \phi\Box\phi d\vol
\end{equation}
for any domain $\mathcal{D}$ in $\mathbb{R}^{1+3}$. Here $\pi^X=\f12 \cL_X m$  is the deformation tensor for the vector field $X$.

In the above energy identity, choose the vector fields $X$, $Y$ and the function $\chi$ as follows:
\[
X=v_+^{\gamma} L+u_+^\gamma \Lb,\quad Y=0,\quad \chi=r^{-1}(v_+^{\ga}-u_+^{\gamma}).
\]
Take the region $\cD=\mathcal{J}^{-}(q)$, which is bounded by the backward light cone $\N^{-}(q)$ and the initial hypersurface. For the above chosen vector field $X$, we can compute that
\[
\nabla_{L}X=\gamma v_+^{\gamma-2}v L,\quad \nabla_{\Lb}X= \gamma u_+^{\gamma-2}u \Lb,\quad \nabla_{e_i}X=r^{-1}(v_+^\gamma-u_+^{\gamma}) e_i.
\]
Then the non-vanishing components of the deformation tensor $\pi_{\mu\nu}^X$ are
\[
\pi^X_{L\Lb}=-\gamma \left(v_+^{\gamma-2}v+u_+^{\gamma-2}u\right),\quad \pi^X_{e_i e_i}=r^{-1}(v_+^{\ga}-u_+^{\gamma}).
\]
Therefore we can compute that
\begin{align*}
&T[\phi]^{\mu\nu}\pi^X_{\mu\nu}=2T[\phi]^{L\Lb}\pi^X_{L\Lb}+T[\phi]^{e_ie_i}\pi^X_{e_ie_i}\\
&=-\f12\ga (v_+^{\gamma-2}v+u_+^{\gamma-2}u)(|\nabb\phi|^2+\frac{2}{p+1}|\phi|^{p+1})+r^{-1}(v_+^{\ga}-u_+^{\gamma})(|\nabb\phi|^2-\pa^\mu \phi \pa_\mu\phi-\frac{2}{p+1}|\phi|^{p+1})\\
&=\left(-\f12\ga(v_+^{\gamma-2}v+u_+^{\gamma-2}u)+r^{-1}(v_+^{\ga}-u_+^{\gamma})\right)|\nabb\phi|^2-r^{-1}(v_+^{\ga}-u_+^{\gamma})\pa^\mu\phi \pa_\mu\phi\\
&\quad +\left(-\f12\ga(v_+^{\gamma-2}v+u_+^{\gamma-2}u)-r^{-1}(v_+^{\ga}-u_+^{\gamma})\right)\frac{2}{p+1}|\phi|^{p+1}.
\end{align*}
 Since the function $\chi$ is spherically symmetric, we can compute that
\begin{align*}
\Box \chi=-r^{-1}L\Lb (r\chi)=-2r^{-1}L\Lb(v_+^\ga-u_+^\ga)=0,\quad r>0.
\end{align*}
At $r=0$ it grows at most $r^{\gamma-3}$. Therefore we can write that
\begin{align*}
&T[\phi]^{\mu\nu}\pi^X_{\mu\nu}+
\chi \pa_\mu\phi \pa^\mu\phi -\f12\Box\chi\cdot|\phi|^2+\chi\phi\Box\phi  \\
&=\left(\chi-\f12\ga(v_+^{\gamma-2}v+u_+^{\gamma-2}u)\right)|\nabb\phi|^2 + \left(\frac{p-1}{p+1}\chi-\frac{(v_+^{\ga-2}v+u_+^{\ga-2}u)\ga}{p+1}\right)|\phi|^{p+1}.
\end{align*}
Denote $f(s)=(1+s^2)^{\frac{\ga}{2}}$. Then $\chi=\frac{f(v)-f(u)}{v-u}$. It can be checked directly that the derivative $f'(s)$ is concave. In particular we conclude that
\begin{align*}
  \chi=\frac{f(v)-f(u)}{v-u}\geq \f12 (f'(v)+f'(u))=\f12 \ga(v_+^{\ga-2}v+u_+^{\ga-2}u).
\end{align*}
Therefore the coefficient of $|\nabb\phi|^2$ is nonnegative. On the other hand, the coefficient of $|\phi|^{p+1}$ has an upper bound 
\begin{align*}
  |\frac{p-1}{p+1}\chi-\frac{(v_+^{\ga-2}v+u_+^{\ga-2}u)\ga}{p+1}|\leq Cv_+^{\ga-1}
\end{align*}
for some constant $C$ depending only on $p$ and $\ga$. We remark here that this coefficient is  nonnegative for the super-conformal case when $p\geq 3$. We use Proposition \ref{prop:spacetime:bd} to control this potential term for the sub-conformal case when the sign is indefinite.

We next compute the boundary integrals on the left hand side of the energy identity \eqref{eq:energy:id}, which consists of the integral on the initial hypersurface $\B_{(0, x_0)}(t_0)$ and the integral on the backward light cone $\mathcal{N}^{-}(q)$. Let's first compute the boundary integral on the initial hypersurface, under the coordinates system $(t, x)$. As the initial hypersurface $\B_{(0, x_0)}(t_0)$ has the volume form $dx$, the contraction reads
\begin{align*}
i_{J^{X, \chi}[\phi ]}d\vol&= T[\phi ]_{0 L}X^L+T[\phi ]_{\Lb 0}X^{\Lb}- \f12 \pa_t\chi |\phi|^2+ \f12\chi \pa_t|\phi|^2\\
&= \f12 v_+^\ga( |L\phi|^2+|\nabb\phi|^2+\frac{2}{p+1}|\phi|^{p+1})-\f12\pa_t\chi \cdot |\phi|^2+\f12\chi \pa_t|\phi|^2 \\
&\quad +\f12 u_+^\ga(|{\Lb}\phi|^2+|\nabb\phi|^2+\frac{2}{p+1}|\phi|^{p+1})\\
&=\f12(u_+^\ga+v_+^\ga)( |\nabb\phi|^2+\frac{2}{p+1}|\phi|^{p+1})+\f12 v_+^\ga r^{-2}|L(r\phi)|^2\\
&\quad +\f12 u_+^\ga r^{-2}|{\Lb}(r\phi)|^2-\div(\om r^{-1}|\phi|^2(u_+^\ga+v_+^\ga)).
\end{align*}
Here $\om=\frac{x}{|x|}$ can be viewed as a vector on $\mathbb{R}^{3}$ and the divergence is taken over the initial hypersurface $\B_{(0, x_0)}(t_0)$. The integral of the divergence term can be computed by using integration by parts. Under the coordinates $\tilde{x}=x-x_0$ on the initial hypersurface, we have
\begin{align*}
\int_{\B_{(0, x_0)}(t_0)} \div (\om r^{-1}|\phi|^2(u_+^\ga+v_+^\ga))dx&=\int_{\B_{(0, x_0)}(t_0)} \div (\om r^{-1}|\phi|^2(u_+^\ga+v_+^\ga))d\tilde{x}\\
&=\int_{\S_{(0, x_0)}(t_0)} \tilde{r}^2 \tilde{\om} \cdot\om r^{-1}|\phi|^2(u_+^\ga+v_+^\ga)d\tilde{\om}.
\end{align*}
In particular we derive that
\begin{equation}
\label{eq:PWE:ex:bxt0}
\begin{split}
 &\int_{\B_{(0, x_0)}(t_0)} i_{J^{X,\chi}[\phi]}d\vol + \int_{\S_{(0, x_0)}(t_0)} \tilde{r}^2 \tilde{\om} \cdot\om r^{-1}|\phi|^2(u_+^\ga+v_+^\ga)d\tilde{\om}\\
&=\f12 \int_{\B_{(0, x_0)}(t_0)}     v_+^{\ga} r^{-2}|L(r\phi)|^2 + (u_+^{\ga}+v_+^{\ga})(|\nabb\phi|^2+\frac{2}{p+1}|\phi|^{p+1})+u_+^{\ga}r^{-2}|\Lb(r\phi)|^2 dx \\
 &\leq C \mathcal{E}_{0, \ga} 
 \end{split}
\end{equation}
for some constant $C$ depending only on $\ga$.

For the boundary integral on the backward light cone $\mathcal{N}^{-}(q)$, we shift to the coordinates centered at the point $q=(t_0, x_0)$.
 Recall the volume form
\[
d\vol=dxdt=d\tilde{x}d\tilde{t}=2\tilde{r}^2 d\tilde{v}d\tilde{u}d\tilde{\om}.
\]
Since the backward light cone $\mathcal{N}^{-}(q)$  can be characterized by $\{\tilde{v}=0\}$ under these new coordinates $(\tilde{t}, \tilde{x})$, we therefore have
\begin{align*}
-i_{J^{X, \chi}[\phi]}d\vol=J_{\tilde{\Lb}}^{X, \chi}[\phi]\tilde{r}^2d\tilde{u}d\tilde{\om}= ( T[\phi]_{\tilde{\Lb}\nu}X^\nu -
\f12(\tilde{\Lb}\chi) |\phi|^2 + \f12 \chi\cdot\tilde{\Lb}|\phi|^2 ) \tilde{r}^2d\tilde{u}d\tilde{\om}.
\end{align*}
For the main quadratic terms, we first can compute that
\begin{align*}
T[\phi]_{\tilde{\Lb}\nu}X^\nu =T[\phi]_{\tilde{\Lb}\tilde{\Lb}}X^{\tilde{\Lb}}+T[\phi]_{\tilde{\Lb}\tilde{L}}X^{\tilde{L}}+T[\phi]_{\tilde{\Lb}\tilde{e}_i}X^{\tilde{e}_i}.
\end{align*}
Since the vector field $X$ is given under the coordinates $(t, x)$, we need to write it under the new null frame $\{\tilde{L}, \tilde{\Lb}, \tilde{e}_1, \tilde{e}_2\}$ centered at the point $q$. Note that
\begin{align*}
\pa_r=\om \cdot \nabla=\om \cdot \tilde{\nabla}=\om\cdot \tilde{\om}\pa_{\tilde{r}}+ \om\cdot (\tilde{\nabla}-\tilde{\om}\pa_{\tilde{r}}).
\end{align*}
Here $\om=\frac{x}{|x|}$, $\nabla=(\pa_{x^1}, \pa_{x^2}, \pa_{x^3})$. Thus we can write that
\begin{align*}
X &=(v_+^\ga+u_+^\ga)\pa_{t}+(v_+^\ga-u_+^\ga)\pa_r\\
&=(v_+^\ga+u_+^\ga)\pa_{\tilde{t}}+(v_+^\ga-u_+^\ga)(\om\cdot \tilde{\om}\pa_{\tilde{r}}+ \om\cdot \tilde{\nabb})\\
&=\f12 \left(u_+^\ga+v_+^\ga+(v_+^\ga-u_+^\ga)\om\cdot \tilde{\om}\right)\tilde{L}+\f12 \left(u_+^\ga+v_+^\ga-(v_+^\ga-u_+^\ga)\om\cdot \tilde{\om}\right)\tilde{\Lb}+(v_+^\ga-u_+^\ga)\om\cdot \tilde{\nabb}
\end{align*}
with $\tilde{\nabb}=\tilde{\nabla}-\pa_{\tilde{r}}$. Thus we can compute the quadratic terms
\begin{align*}
T[\phi]_{\tilde{\Lb}\nu}X^\nu 
=&\left((1-\tau)v_+^\ga+(1+\tau)u_+^\ga\right)|{\tilde{\Lb}}\phi|^2 +\left((1+\tau)v_+^\ga+(1-\tau)u_+^\ga\right)(|\tilde{\nabb}\phi|^2+\frac{2}{p+1}|\phi|^{p+1})\\
&+2 (v_+^\ga-u_+^\ga) ({\tilde{\Lb}}\phi) (\om\cdot \tilde{\nabb})\phi.
\end{align*}
 Here recall that $\tau=\om\cdot \tilde{\om}$.
 It turns out that these terms are nonnegative but we need to estimate them together with the lower order terms arising from the function $\chi$.
We compute that
\begin{align*}
&\tilde{\Lb}(r)=-\pa_{\tilde{r}}(r)=-\tilde{\om}_i\pa_i(r)=-\tilde{\om}\cdot \om =-\tau,\\
&\tilde{\nabb}(r)=(\tilde{\nabla}-\tilde{\om}\pa_{\tilde{r}})(r)=\om-\tilde{\om}\tau.
\end{align*}
Therefore we can write
\begin{align*}
&-\f12 r^2 (\tilde{\Lb}{\chi})|\phi|^2+\f12 r^2\chi \tilde{\Lb}|\phi|^2=(r\chi)( {\tilde{\Lb}}(r\phi)+\tau \phi) \phi-\f12(r\tilde{\Lb}(r\chi)+\tau r\chi)|\phi|^2,\\
&r^2|\tilde{\Lb}\phi|^2=|{\tilde{\Lb}}(r\phi)-\tilde{\Lb}(r)\phi|^2=|{\tilde{\Lb}}(r\phi)|^2+\tau^2|\phi|^2+2 {\tilde{\Lb}}(r\phi) \tau\phi,\\
& r^2|\tilde{\nabb}\phi|^2
=|\tilde{\nabb}(r\phi)|^2+(1-\tau^2)|\phi|^2-2(\om-\tilde{\om}\tau)\cdot\tilde{\nabb}(r\phi) \phi,\\
& r^2 ({\tilde{\Lb}}\phi)  (\om\cdot \tilde{\nabb})\phi={\tilde{\Lb}}(r\phi) (\om \cdot \tilde{\nabb})(r\phi)-\tau(1-\tau^2)|\phi|^2+\phi \tau(\om\cdot \tilde{\nabb})(r\phi) -(1-\tau^2){\tilde{\Lb}}(r\phi)\phi.
\end{align*}
Notice that
\[
\om\cdot \tilde{\nabb}=\om\cdot (\tilde{\om}\times \tilde{\nabla})=\om\times \tilde{\om}\cdot \tilde{\nabb}.
\]
Since $v_+\geq u_+$, we therefore can show that the quadratic terms are nonnegative
\begin{align*}
&\left((1-\tau)v_+^\ga+(1+\tau)u_+^\ga\right)|\tilde{\Lb}(r\phi)|^2+\left((1+\tau)v_+^\ga+(1-\tau)u_+^\ga\right)|\tilde{\nabb}(r\phi)|^2+2 (v_+^\ga-u_+^\ga){\tilde{\Lb}}(r\phi) (\om \cdot \tilde{\nabb})(r\phi)\\
&\geq 2\sqrt{(v_+^{\ga}+u_+^{\ga})^2-\tau^2(v_+^{\ga}-u_+^{\ga})^2}|\tilde{\nabb}(r\phi)||\tilde{\Lb}(r\phi)|-2 (v_+^\ga-u_+^\ga)\sqrt{1-\tau^2}|{\tilde{\Lb}}(r\phi)| |\tilde{\nabb}(r\phi)|\\
& \geq 0.
\end{align*}
For the other lower order terms, we compute  that
\begin{align*}
&\left((1-\tau)v_+^\ga+(1+\tau)u_+^\ga\right)(\tau^2|\phi|^2+2{\tilde{\Lb}}(r\phi)\tau\phi )+(r\chi)( {\tilde{\Lb}}(r\phi)+\tau \phi) \phi-\f12(r\tilde{\Lb}(r\chi)+\tau r\chi)|\phi|^2\\
&+\left((1+\tau)v_+^\ga+(1-\tau)u_+^\ga\right)\big((1-\tau^2)|\phi|^2-2(\om-\tilde{\om}\tau)\tilde{\nabb}(r\phi) \phi\big)\\
&+(v_+^\ga-u_+^\ga)\big(-2\tau(1-\tau^2)|\phi|^2+2\tau \phi (\om\cdot \tilde{\nabb})(r\phi)-2\phi (1-\tau^2){\tilde{\Lb}}(r\phi) \big)\\
&=(-\f12 r\tilde{\Lb}(r\chi)+v_+^\ga+u_+^\ga)|\phi|^2
+2(v_+^\ga+u_+^\ga)(\tau {\tilde{\Lb}}-\om\cdot \tilde{\nabb})(r\phi) \phi\\
&=-r^2\tilde{r}^{-1} \tilde{\Om}_{ij}(r^{-3}(v_+^\ga+u_+^\ga) \om_j\tilde{\om}_i |r\phi|^2)+\tilde{r}^{-2}r^2\tilde{\Lb}(r^{-1}\tau\tilde{r}^2(v_+^\ga+u_+^\ga) |\phi|^2)\\
&+(-\f12 r\tilde{\Lb}(r\chi)+v_+^\ga+u_+^\ga)|\phi|^2-\tilde{r}^{-2}r^2\tilde{\Lb}(r^{-3}\tau\tilde{r}^2(v_+^\ga+u_+^\ga)) |r\phi|^2+r^2 \tilde{r}^{-1}\tilde{\Om}_{ij}(r^{-3}(v_+^\ga+u_+^\ga) \om_j\tilde{\om}_i) |r\phi|^2.
\end{align*}
Note that
\begin{align*}
&\tilde{r}^{-1}\tilde{\Om}_{ij}(r^{-3}\om_j\tilde{\om}_i)=-2r^{-4}(1-2\tau^2)-2\tau \tilde{r}^{-1}r^{-3},\\
&\tilde{r}^{-2}r^4\tilde{\Lb}(r^{-3}\tilde{r}^2\tau)=4\tau^2-1-2r\tilde{r}^{-1}\tau.
\end{align*}
Thus the coefficients of $|\phi|^2$ in the last line in the previous equation verify
\begin{align*}
&(-\f12 r\tilde{\Lb}(r\chi)+v_+^\ga+u_+^\ga)-\tilde{r}^{-2}r^4\tilde{\Lb}(r^{-3}\tau\tilde{r}^2(v_+^\ga+u_+^\ga)) +r^4 \tilde{r}^{-1} \tilde{\Om}_{ij}(r^{-3}(v_+^\ga+u_+^\ga) \om_j\tilde{\om}_i) \\
&=-r(\pa_t-\tilde{\om}\cdot \nabla)(v_+^\ga-u_+^\ga)-r\tau (\pa_t-\tilde{\om}\cdot \nabla)(u_+^\ga+v_+^\ga)+r(\pa_r-\tau \tilde{\om}\cdot \nabla)(u_+^\ga+v_+^\ga)\\
&\quad +(u_+^\ga+v_+^\ga)\left(1-(4\tau^2-1-2r\tilde{r}^{-1}\tau)-2(1-2\tau^2)+2\tau \tilde{r}^{-1}r\right)\\
&=r(\pa_t+\pa_r)u_+^\ga+r(\pa_r-\pa_t)v_+^\ga-\tau r(\pa_t+\pa_r)u_+^\ga-\tau r(\pa_t-\pa_r)v_+^\ga=0.
\end{align*}
The above computations imply that the lower order terms can be written as a divergence form and hence can be estimated by using integration by parts:
\begin{align*}
  &\int_{\mathcal{N}^{-}(q)}\big(-r^2\tilde{r}^{-1} \tilde{\Om}_{ij}(r^{-3}(v_+^\ga+u_+^\ga) \om_j\tilde{\om}_i |r\phi|^2)+\tilde{r}^{-2}r^2\tilde{\Lb}(r^{-1}\tau\tilde{r}^2(v_+^\ga+u_+^\ga) |\phi|^2)\big)r^{-2}\tilde{r}^2 d\tilde{u}d\tilde{\om}\\
  &= \int_{\S_{(0, x_0)}(t_0)}r^{-1}\tau \tilde{r}^2 (u_+^\ga+v_+^\ga)|\phi|^2d\tilde{\om}.
\end{align*}
This term is an integral on a sphere in the initial hypersurface and cancels the one arising from the boundary integral on $\B_{(0, x_0)}(t_0)$.
Keeping the potential part and discarding the quadratic terms which are nonnegative, we therefore derive that
\begin{equation*}
\begin{split}
&\frac{2}{p+1}\int_{\mathcal{N}^{-}(q)}((1+\tau)v_+^\ga+(1-\tau)u_+^\ga ) |\phi|^{p+1} \tilde{r}^2d\tilde{u}d\tilde{\om}\\
&\leq -\int_{\mathcal{N}^{-}(q)}i_{J^{X,\chi}[\phi]}d\vol+ \int_{\S_{(0, x_0)}(t_0)}r^{-1}\tau \tilde{r}^2 (u_+^\ga+v_+^\ga)|\phi|^2d\tilde{\om}.
\end{split}
\end{equation*}
Combining this estimate with \eqref{eq:PWE:ex:bxt0} and by using the uniform spacetime bound of Proposition \ref{prop:spacetime:bd}, we then derive that
\begin{align*}
  \frac{2}{p+1}\int_{\mathcal{N}^{-}(q)}((1+\tau)v_+^\ga+(1-\tau)u_+^\ga ) |\phi|^{p+1} 
 & \leq \iint_{\mathcal{J}^{-}(q)}\left|\frac{p-1}{p+1}\chi-\frac{(v_+^{\ga-2}v+u_+^{\ga-2}u)\ga}{p+1}\right||\phi|^{p+1}\\
  &\quad +\int_{\pa\mathcal{J}^{-}(q)}i_{J^{X, \chi}[\phi]}d\vol\\
  &\leq C \mathcal{E}_{0, \ga+\ep} 
\end{align*}
for some constant $C$ depending only on $\ga$, $p$ and $0<\ep<p-1-\ga$. Since $v_+\geq u_+$, the proposition then follows by letting $0<\ep<\ga_0-\ga$.  
\end{proof}

\section{Decay estimates for the solution in the exterior region}

In this section, we use of the weighted energy flux bound derived in the previous section to show the pointwise decay estimates for the solution when $p>2$ in the exterior region $\{t+2\leq |x|\}$.
We need the following integration lemma.
 \begin{Lem}
\label{lem:integration:ex:ab}
  Assume $1<\ga<2$ and $\a$, $\b$ nonnegative such that $\b+\a\ga>2$, $\alpha\neq 1$. Fix $q=(t_0, x_0)$ in the exterior region $\{t+2\leq |x|\}$. For the $2$-sphere $\S_{(t_0-\tilde{r}, x_0)}(\tilde{r})$ on the backward light cone $\mathcal{N}^{-}(q)$, we have
  \begin{equation}
    \label{eq:integration:ex:ab}
    \begin{split}
    &\int_{\S_{(t_0-\tilde{r}, x_0)}(\tilde{r})} ((1+\tau)r^{\ga}+(r_0-t_0)^\ga)^{-\a}r^{-\b}d\tilde{\om}\\
    &\leq C (r_0-\tilde{r})^{2-\b-\ga+\ep} r_0^{-2}\left((r_0-\tilde{r})^{(1-\a)\ga}+(r_0-t_0)^{(1-\a)\ga}\right)
    \end{split}
  \end{equation}
  for some constant $C$ depending only on $\ep$, $\ga$, $\a$ and $\b$.
  Here $\tau=\om\cdot \tilde{\om}$, $r_0=|x_0|$ and $0\leq \tilde{r}\leq t_0<r_0$.
 \end{Lem}
 \begin{proof}
Denote $s=-\om_0\cdot \tilde{\om}$ and $\om_0=r_0^{-1}x_0$. By definition, we have
  \begin{align*}
    &r^2=|x_0+\tilde{x}|^2=\tilde{r}^2+r_0^2-2r_0\tilde{r}s=(\tilde{r}-r_0s)^2+(1-s^2)r_0^2,\\
    &(1+\tau)r=r+r\om\cdot \tilde{\om}=r+(\tilde{x}+x_0)\cdot \tilde{\om}=r+\tilde{r}-r_0s.
  \end{align*}
We can write the integral as
  \begin{align*}
    &\int_{\S_{(t_0-\tilde{r}, x_0)}(\tilde{r})} ((1+\tau)r^{\ga}+(r_0-t_0)^\ga)^{-\a}r^{-\b}d\tilde{\om} =2\pi\int_{-1}^1 r^{-\b}(r^{\ga-1}(r+\tilde{r}-r_0s)+(r_0-t_0)^{\ga})^{-\a}ds.
  \end{align*}
 Since $r_0\geq t_0+2$ and $\tilde{r}\leq t_0$, the estimate \eqref{eq:integration:ex:ab} trivially holds when $t_0+r_0\leq 20$. Moreover for the case when $t_0\leq \frac{9}{10}r_0$, note that $r\geq r_0-t_0\geq \frac{1}{10}r_0$. We thus can bound that
   \begin{align*}
    \int_{\S_{(t_0-\tilde{r}, x_0)}(\tilde{r})} ((1+\tau)r^{\ga}+(r_0-t_0)^\ga)^{-\a}r^{-\b}d\tilde{\om}
    \leq 4\pi 10^{\b+\a\ga} r_0^{-\b-\a\ga}.
  \end{align*}
  Hence it remains to prove the Lemma when $t_0+r_0\geq 20$ and $\frac{9}{10}r_0 \leq t_0\leq r_0$.

For the integral on $s\leq 0$, we bound that
  \begin{align*}
    r\geq r_0,\quad r+\tilde{r}-r_0s\geq r+\tilde{r}\geq r_0.
  \end{align*}
  Thus
  \begin{align*}
    \int_{-1}^0 r^{-\b}(r^{\ga-1}(r+\tilde{r}-r_0s)+(r_0-t_0)^{\ga})^{-\a}ds\leq r_0^{-\b-\a\ga}.
  \end{align*}
  Define $s_0=1-(1-\tilde{r}r_0^{-1})^2$. On the interval $[0, s_0]$, note that
  \begin{align*}
    \sqrt{1-s} \ r_0\geq r_0-\tilde{r}.
  \end{align*}
  Therefore, we can show that
  \begin{align*}
    \tilde{r}-r_0s \leq r_0(1-s)\leq r_0\sqrt{1-s},\quad r_0s-\tilde{r}\leq r_0-\tilde{r}\leq r_0\sqrt{1-s}
  \end{align*}
  as $\tilde{r}\leq t_0<r_0$. This in particular implies that
  \begin{align*}
    \sqrt{1-s}\ r_0\leq r\leq \sqrt{2(1-s)}\ r_0,\quad \sqrt{(\tilde{r}-r_0s)^2+(1-s^2)r_0^2}+\tilde{r}-r_0s\geq \frac{1}{3}\sqrt{1-s}r_0.
  \end{align*}
  Here the second inequality follows from the inequality
  \[
  \sqrt{a^2+b^2}+b\geq (\sqrt{2}-1)|a|,\quad \forall |b|\leq |a|
  \]
  together with the bound $|\tilde{r}-r_0s|\leq \sqrt{1-s}r_0\leq \sqrt{1-s^2}r_0$.

  Therefore on the interval $[0, s_0]$, we can estimate that
  \begin{align*}
    \int_{0}^{s_0} r^{-\b}(r^{\ga-1}(r+\tilde{r}-r_0s)+(r_0-t_0)^{\ga})^{-\b} ds
    &\leq 3^{\a} \int_{0}^{s_0}(\sqrt{1-s} r_0)^{-\b-\a\ga} ds\\
    &\leq \frac{2 \times 3^{\a} }{\b+\a\ga-2}r_0^{-2}(r_0-\tilde{r})^{2-\b-\a\ga}.
  \end{align*}
  Here we used the assumption $\b+\a\ga> 2$.

  Finally on the interval $[s_0, 1]$, notice that
  \begin{align*}
    2r\geq r_0s-\tilde{r}+\sqrt{1-s}r_0=r_0-\tilde{r}+(\sqrt{1-s}-(1-s))r_0\geq r_0-\tilde{r}.
  \end{align*}
  Moreover
  \begin{align*}
    r+\tilde{r}-r_0s=\frac{(1-s^2)r_0^2}{r+r_0s-\tilde{r}}\geq\frac{(1-s)r_0^2}{4(r_0-\tilde{r})}.
  \end{align*}
  Therefore we can estimate that
  \begin{align*}
    &\int_{s_0}^{1} r^{-\b}(r^{\ga-1}(r+\tilde{r}-r_0s)+(r_0-t_0)^{\ga})^{-\a} ds\\
    &\leq 2^{\b} \int_{s_0}^{1}\big((r_0-t_0)^{\ga}+2^{\ga-3}(r_0-\tilde{r})^{\ga-2}(1-s)r_0^2\big)^{-\a}(r_0-\tilde{r})^{-\b}ds\\
    &= 2^{\b+3-\ga}(\a-1)^{-1}(r_0-\tilde{r})^{2-\b-\ga} r_0^{-2}\left((r_0-t_0)^{(1-\a)\ga}-((r_0-t_0)^{\ga}+2^{\ga-3}(r_0-\tilde{r})^{\ga})^{1-\a}\right)\\
    &\leq C  (r_0-\tilde{r})^{2-\b-\ga } r_0^{-2}\left((r_0-\tilde{r})^{(1-\a)\ga}+(r_0-t_0)^{(1-\a)\ga}\right)
  \end{align*}
  for some constant $C $ depending only on $\a$, $\b$ and $\ga$. Here we note that $\alpha\neq 1$.
 Since
  \begin{align*}
    r_0^{-\a\ga-\b}\leq (r_0-\tilde{r})^{2-
    \b-\ga } r_0^{-2}\left((r_0-\tilde{r})^{(1-\a)\ga}+(r_0-t_0)^{(1-\a)\ga}\right)
  \end{align*}
in view of the assumption $\b+\a\ga> 2$, we thus conclude the Lemma.
\end{proof}
We now use the above integration lemma as well as the weighted energy estimate of Proposition \ref{prop:EF:cone:NW:3d} to study the asymptotic behavior of the solution in the pointwise sense in the exterior region $\{t+2\leq |x|\}$. First recall the representation formula for wave equations
\begin{equation}
\label{eq:rep4phi:ex}
\begin{split}
4\pi\phi(t_0, x_0)&=\int_{\tilde{\om}}t_0  \phi_1(x_0+t_0\tilde{\om})d\tilde{\om}+\pa_{t_0}\big(\int_{\tilde{\om}}t_0  \phi_0(x_0+t_0\tilde{\om})d\tilde{\om}   \big)-\int_{\mathcal{N}^{-}(q)}|\phi|^{p-1} \phi \ \tilde{r} d\tilde{r}d\tilde{\om}
\end{split}
\end{equation}
for any point $q=(t_0, x_0)\in\mathbb{R}^{1+3}$. The first two terms are the associated linear evolution part which has been well understood. The nonlinearity will be controlled by using the weighted energy estimate of Proposition \ref{prop:EF:cone:NW:3d} together with the above integration lemma. We divide into two cases according to the range of the power $p$.

First note that the backward light cone $\mathcal{N}^{-}(q)$ locates in the region $\{t-|x|\leq t_0- r_0\}$. Hence when $t_0+2\leq r_0=|x_0|$, we always have 
\begin{align*}
u_+=\sqrt{1+u^2}\geq \sqrt{1+\frac{(t_0-r_0)^2}{4}}\geq \frac{r_0-t_0+1}{4},\quad v_+\geq \frac{r}{2}.
\end{align*}
In view of Proposition \ref{prop:EF:cone:NW:3d}, we hence conclude that 
\begin{equation}
\label{eq:Eflux:ex:EF:rW:a}
\begin{split}
&\int_{\mathcal{N}^{-}(q)}((1+\tau)r^{\ga}+(r_0-t_0)^{\ga})|\phi|^{p+1} d\si 
\leq C \mathcal{E}_{0, \ga} 
\end{split}
\end{equation}
for all $0\leq \ga\leq \min\{p-1, 2\}$ and $t_0+2\leq r_0=|x_0|$.

For the case when 
$\frac{1+\sqrt{17}}{2}<p<5$, we can choose $\ga_0$ such that
\[
\max\{\frac{4}{p-1}-1, 1\}<\ga_0<\min\{p-1, 2\}.
\]
In view of the representation formula \eqref{eq:rep4phi:ex}, the linear evolution part decays as follows
\begin{align}
\label{eq:lindecay:ex}
  |\int_{\tilde{\om}}t_0  \phi_1(x_0+t_0\tilde{\om})d\tilde{\om}+\pa_{t_0}\big(\int_{\tilde{\om}}t_0  \phi_0(x_0+t_0\tilde{\om})d\tilde{\om}   \big)|
  &\les r_0^{-1} (r_0-t_0)^{-\frac{\ga_0-1}{2}}\sqrt{\mathcal{E}_{1, \ga_0}  }
\end{align}
by using the standard vector field method. Here $\ga_0>1$ and $r_0=|x_0|\geq t_0+2$.

Define nonnegative small constant $\ep_p$ such that 
\begin{equation*}
\ep_p=
\begin{cases}
0,\quad p<3;\\
\min\{\frac{5-p}{10}, \frac{\ga_0-1}{4}\},\quad 3\leq p<5.
\end{cases}
\end{equation*}
For the nonlinear part, apply the integration lemma \ref{lem:integration:ex:ab} with $\a=\frac{p-1+\ep_p}{2-\ep_p}$, $\b=\frac{(p+1)(1-\ep_p)}{2-\ep_p}$. In view of the above weighted energy estimate \eqref{eq:Eflux:ex:EF:rW:a}, we can bound that
\begin{equation}
\label{eq:bdN:largep}
\begin{split}
  &|\int_{\mathcal{N}^{-}(q)}\Box \phi \ \tilde{r} d\tilde{r}d\tilde{\om}|\\
  &\leq \left(\int_{\mathcal{N}^{-}(q)}((1+\tau)r^{\ga_0}+(r_0-t_0)^{\ga_0})|\phi|^{p+1}\ \tilde{r}^{2} d\tilde{r}d\tilde{\om}\right)^{\frac{p-1+\ep_p}{p+1}}\\
  &\quad \cdot \left(\int_{\mathcal{N}^{-}(q)}((1+\tau)r^{\ga_0}+(r_0-t_0)^{\ga_0})^{-\frac{p-1+\ep_p}{2-\ep_p}}|\phi|^{\frac{(1-\ep_p)(p+1)}{2-\ep_p}}\ \tilde{r}^{\frac{3-p-2\ep_p}{2-\ep_p}} d\tilde{r}d\tilde{\om}\right)^{\frac{2-\ep_p}{p+1}}\\ 
  &\les \left(\int_{0}^{t_0}(\sup\limits_{x}|r\phi|^{\b}) (r_0-\tilde{r})^{2-\b-\ga_0} r_0^{-2}\left((r_0-\tilde{r})^{(1-\a)\ga_0}+(r_0-t_0)^{( 1-\a)\ga_0}\right)\tilde{r}^{\frac{3-p-2\ep_p}{2-\ep_p}} d\tilde{r} \right)^{\frac{2-\ep_p}{p+1}}.
  \end{split}
\end{equation}
Here and in the sequel of this section, the implicit constant in $\les$ may also rely on the zeroth order weighted energy $\mathcal{E}_{0, \ga_0} $.
Define the function
\begin{equation*}
\mathcal{M}(t_0)=\sup\limits_{|x|\geq t_0+2} |u_+^{\frac{\ga_0-1}{2}}r\phi|^{\frac{p+1}{2-\ep_p}}
\end{equation*}
and
\begin{equation*}
  f(t_0, r_0, \tilde{r})=r_0^{\frac{p+1}{2-\ep_p}}(r_0-\tilde{r})^{2-\b-\ga_0} r_0^{-2}\left((r_0-\tilde{r})^{(1-\a)\ga_0}+(r_0-t_0)^{(1-\a)\ga_0}\right)\tilde{r}^{\frac{3-p-2\ep_p}{2-\ep_p}} .
\end{equation*}
Here $u_+=1+\f12|t-r|$. Since on the cone $\mathcal{N}^{-}(q)$, we always have $u_+\geq \f12 (r_0-t_0)$. 
Then the previous inequality leads to
\begin{align}
\label{eq:001}
  \mathcal{M}(t_0)\les (\mathcal{E}_{1, \ga_0}[\phi] )^{\frac{p+1}{2(2-\ep_p)}}+\int_{0}^{t_0}\mathcal{M}^{1-\ep_p}(t_0-\tilde{r})f(t_0, r_0, \tilde{r}) d\tilde{r}.
\end{align}
When $p<3$, by definition we have $\ep_p=0$. In view of the assumption $(p-1)(\ga_0+1)>4$, we also have
$$\frac{5-p}{2} + \frac{1-p}{2}\ga_0<  0.$$ 
Since $\tilde{r}\leq t_0\leq r_0-2$, we thus can bound that
\begin{align*}
   f(t_0, r_0, \tilde{r})\les (r_0-\tilde{r})^{\frac{3-p}{2}-\ga_0 } r_0^{\frac{p-3}{2}} (r_0-\tilde{r})^{\frac{3-p}{2}\ga_0} \tilde{r}^{\frac{3-p}{2}}\leq  (2+t_0-\tilde{r})^{\frac{5-p}{2}+\frac{1-p}{2}\ga_0 -1}  .
\end{align*}
By using Gronwall's inequality, we conclude that 
\begin{align*}
\mathcal{M}(t_0)\les (\mathcal{E}_{1, \ga_0}  )^{\frac{p+1}{4}},\quad \forall t_0\geq 0. 
\end{align*}
When $p\geq 3$, for the range when $\tilde{r}\leq \frac{t_0}{2}+1$, we have 
\begin{align*}
   f(t_0, r_0, \tilde{r})&\leq 2(r_0-\tilde{r})^{\frac{3-p+(p-1)\ep_p}{2-\ep_p}-\ga_0}   (r_0-t_0)^{\frac{3-p-2\ep_p}{2-\ep_p}\ga_0} (r_0^{-1}\tilde{r})^{\frac{3-p-2\ep_p}{2-\ep_p}} \\
   & \les (2+t_0 )^{\frac{(p+1)\ep_p}{2-\ep_p}-\ga_0} \tilde{r}^{\frac{3-p-2\ep_p}{2-\ep_p}}.
\end{align*}
When $\frac{t_0}{2}+1\leq \tilde{r}\leq t_0$, we can bound that 
\begin{align*}
   f(t_0, r_0, \tilde{r})&\les (r_0-\tilde{r})^{\frac{3-p+(p-1)\ep_p}{2-\ep_p}-\ga_0}    ((r_0-t_0)r_0^{-1} t_0)^{\frac{3-p-2\ep_p}{2-\ep_p}} \\
   & \les   (2+t_0-\tilde{r})^{ \frac{3-p+(p-1)\ep_p}{2-\ep_p}-\ga_0}.
\end{align*}
By the choice of $\ep_0$ when $p\geq 3$, we note that 
\begin{align*}
 \frac{3-p+(p-1)\ep_p}{2-\ep_p}-\ga_0+1<4\ep_0+1-\ga_0\leq 0. 
\end{align*}
In particular, we always have 
\begin{align*}
  \int_0^{t_0}f(t_0,r_0, \tilde{r})d\tilde{r}  \les 1. 
\end{align*}
 As $0<\ep_p<1$ when $p\geq 3$, in view of inequality \eqref{eq:001} and the uniform integrability of the function $f(t_0, r_0, \tilde{r})$, by using a bootstrap argument,  we conclude that
\begin{align*}
   \mathcal{M}(t_0)\les (\mathcal{E}_{1, \ga_0}  )^{\frac{p+1}{2(2-\ep_p)}}.
\end{align*}
The decay estimate for $\phi$ then follows from the definition of $\mathcal{M}(t_0)$.

Next we investigate the asymptotic decay for the solution in the exterior region when $2<p\leq \frac{1+\sqrt{17}}{2}$. For this case we assume that 
$1<\ga_0<p-1\leq \frac{\sqrt{17}-1}{2}$. In particular we have 
$$p\ga_0>2, \quad (p-1)(\ga_0+1)>2,\quad \frac{(p-1)\ga_0}{5-p}<\frac{(p-1)^2}{5-p}<1.$$
In view of the representation formula \eqref{eq:rep4phi:ex}, the linear part verifies the same decay estimate as \eqref{eq:lindecay:ex} since $1<\ga_0<2$. 
 To control the nonlinear term, we split the integral on the backward light cone $\mathcal{N}^{-}(q)$ into two parts: the first part is restricted to time interval $[\f12 t_0^{\frac{(p-1)\ga_0}{5-p}}, t_0]$, on which we use the same argument as above and the second part can be directly estimated by the weighted energy flux. More precisely, for the first part, we estimate that
\begin{align*}
  &|\int_{\mathcal{N}^{-}(q)\cap\{\tilde{r}\geq \f12 t_0^{\frac{p-1}{5-p}\ga_0}\}}\Box \phi \ \tilde{r} d\tilde{r}d\tilde{\om}|\\
  &\les \left(\int_{\mathcal{N}^{-}(q)}((1+\tau)r^{\ga_0}+(r_0-t_0)^{\ga_0})|\phi|^{p+1}\ \tilde{r}^{2} d\tilde{r}d\tilde{\om}\right)^{\frac{p}{p+1}}\\
  &\quad \cdot \left(\int_{\mathcal{N}^{-}(q)\cap\{\tilde{r}\geq \f12 t_0^{\frac{p-1}{5-p}\ga_0}\}}((1+\tau)r^{\ga_0}+(r_0-t_0)^{\ga_0})^{-p} \tilde{r}^{1-p} d\tilde{r}d\tilde{\om}\right)^{\frac{1}{p+1}}\\
  &\les  \left(\int_{\f12 t_0 ^{\frac{p-1}{5-p}\ga_0}}^{t_0}  (r_0-\tilde{r})^{2-\ga_0}   r_0^{-2}(r_0-t_0)^{(1-p)\ga_0} \tilde{r}^{1-p} d\tilde{r} \right)^{\frac{1}{p+1}}\\
  &\les r_0^{-\frac{3+(p-2)^2}{(p+1)(5-p)}\ga_0}(r_0-t_0)^{\frac{(1-p)\ga_0}{p+1}}.
\end{align*}
Here we used Lemma \ref{lem:integration:ex:ab} together with the weighted energy decay estimate \eqref{eq:Eflux:ex:EF:rW:a}. For the integral on $[0, \f12 t_0^{\frac{(p-1)\ga_0}{5-p}}]$, we estimate the nonlinearity the same way as in
\eqref{eq:bdN:largep}.  We therefore can derive that
\begin{align*}
  |\int_{\mathcal{N}^{-}(q) }\Box \phi \ \tilde{r} d\tilde{r}d\tilde{\om}|
  \les  &r_0^{-\frac{3+(p-2)^2}{(p+1)(5-p)}\ga_0}(r_0-t_0)^{\frac{(1-p)\ga_0}{p+1}}\\
 &+ \left(\int_{0}^{ \f12 t_0^{\frac{(p-1)\ga_0}{5-p}}}(\sup\limits_{x}|r\phi|^{\frac{p+1}{2}}) (r_0-\tilde{r})^{1-\frac{(p-1)(\ga_0+1)}{2}} r_0^{-2} \tilde{r}^{\frac{3-p}{2}} d\tilde{r} \right)^{\frac{2}{p+1}}.
\end{align*}
Here we note that $\ep_p=0$ as $p<3$ and $r_0-\tilde{r}\geq r_0-t_0$.
  Since $$\tilde{r}\leq \f12 t_0^{\frac{(p-1)\ga_0}{5-p}}\leq \max\{\f12 t_0, 1\},\quad t_0\leq r_0-2,$$
   we in particular have that $r_0\les r$, $r_0\les r_0-\tilde{r}$. Define
\begin{align*}
  \mathcal{M}_1(t)=\sup\limits_{|x|\geq t+2}( |\phi(t, x)| |x|^{\frac{3+(p-2)^2}{(p+1)(5-p)}\ga_0 }(|x|-t)^{\frac{(p-1)\ga_0}{p+1}})^{\frac{p+1}{2}}.
\end{align*}
From the representation formula \eqref{eq:rep4phi:ex} and the linear decay estimate \eqref{eq:lindecay:ex}, we then derive that
\begin{align*}
  \mathcal{M}_1(t_0)&\les 1+\int_0^{\f12 t_0^{\frac{(p-1)\ga_0}{5-p}}}\mathcal{M}_1(t_0-\tilde{r}) r_0^{\frac{p+1}{2}+1-\frac{(p-1)(\ga_0+1)}{2}-2}\tilde{r}^{\frac{3-p}{2}}d\tilde{r}\\
  & \les  1+\int_0^{\f12 t_0^{\frac{(p-1)\ga_0}{5-p}}}\mathcal{M}_1(t_0-\tilde{r}) (2+t_0)^{ -\frac{(p-1)\ga_0 }{2} } \tilde{r}^{\frac{3-p}{2}}d\tilde{r}.
\end{align*}
Here we note that
\begin{align*}
  \frac{3+(p-2)^2}{(p+1)(5-p)}\ga_0\leq 1,\quad \frac{3+(p-2)^2}{(p+1)(5-p)}\ga_0+\frac{(p-1)\ga_0}{p+1}\leq \frac{\ga_0+1}{2}.
\end{align*}
By using Gronwall's inequality, we conclude that
\begin{align*}
  \mathcal{M}_1(t_0)\les 1.
\end{align*}
The pointwise decay estimate for $\phi$ for the case when $2<p\leq \frac{1+\sqrt{17}}{2}$ then follows.
We hence finished the proof for Theorem \ref{thm:main} in the exterior region.

\bigskip

The method for proving the decay estimates in the interior region is similar to the above argument for deriving the decay estimates for the solution in the exterior region after conformal transformation. The nonlinearity will be controlled by the weighted energy flux through backward light cones. To use Gronwall's inequality, one needs first bound the linear evolution with prescribed data, which, in the interior region, will be the data on the hyperboloid $\Hy$. We rely on the representation formula together with the uniform weighted energy flux bound through backward light cones.

Define the region enclosed by the hyperboloid $\Hy$
\begin{align*}
\mathbf{D}:=\left\{(t, x)|(t^*)^2-|x|^2\geq (R^*)^{-1} t^*\right\},\quad \mathbf{D}^+=\mathbf{D}\cap\{t\geq 0\}, \quad R^*=\frac{3}{5},\quad t^*=t+3.
\end{align*}
Let $\phi^{lin}_{H}$ be the linear evolution in $\mathbf{D}$, that is,
\begin{align*}
  \Box \phi^{lin}_{H}=|\phi|^{p-1}\phi (1-\mathbf{1}_{\mathbf{D}^+}),\quad \phi^{lin}_H(0, x)=\phi_0,\quad \pa_t \phi_H^{lin}(0, x)=\phi_1,
\end{align*}
where $\mathbf{1}_{\mathbf{D}^+}$ stands for the characteristic function of the set $\mathbf{D}^+$. We see that $\phi^{lin}_H$ coincides with $\phi$ in the region $(\mathbb{R}^{3}\backslash\mathbf{D})\cap\{t\geq 0\}$. The decay estimates for $\phi^{lin}$ rely on the following integration Lemma. 

\def\Nd{\mathcal{N}^{-}(q)\backslash \mathbf{D}}

\begin{Lem}
\label{Lem:NLW:def:3d:ex:ID:H}
Let $q=(t_0, x_0)\in\mathbb{R}^{1+3}$ such that $t_0 \geq \max\{|x_0|+10, 20\}$. Assume that $l$, $\a$, $0\leq \b<\min\{1, \a\}$ are nonnegative constants such that $l\ga+\a+\b>2.$
Then we have 
\begin{equation}
\label{eq:NLW:def:3d:ex:ID:H} 
\begin{split}
&\int_{\mathcal{N}^{-}(q) \backslash \mathbf{D} } ((1+\tau)v_+^{\ga}+u_+^\ga)^{-l}v_+^{-\a} u_+^{-\b} \tilde{r}^{k}d\tilde{r}d\tilde{\om}\\
&\les  t_0^{k-2} (u_0^{3-l\ga-\a-\b}+t_0^{3-l\ga-\a-\b}) +t_0^{-l\ga-\a}u_0^{-\b}(t_0^{k+1}+u_0^{k+1}+u_0^{1-l}t_0^{k+l}). 
\end{split}
\end{equation}
Here the tilde components are the associated quantities under the coordinates centered at $q$ and $\tau=\om \cdot \tilde{\om}$, $u_0=\frac{t_0-r_0}{2}$, $v_0=\frac{t_0+r_0}{2}$, $r_0=|x_0|$.
\end{Lem}
\begin{proof}
Denote $s=-\om_0\cdot \tilde{\om}$ with $\om_0=r_0^{-1}x_0$, $r_0=|x_0|$. Recall that
  \begin{align*}
    &r^2=|x|^2=(\tilde{r}-r_0s)^2+(1-s^2)r_0^2,\quad r\tau=\tilde{r}-r_0s,\quad u_0=\frac{t_0-r_0}{2},\quad v_0=\frac{t_0+r_0}{2}.
  \end{align*}
  By the assumption, we in particular have $u_0\geq 10$.

On $\Nd$, $\tilde{r}$ and $s$ have to verify the relation
\begin{align*}
  r^2 \geq (t+3)^2 -\frac{5}{3}(t+3) \geq t^2+4,
\end{align*}
that is,
\begin{align*}
 s\leq s_* =\frac{r_0^2-t_0^2+2t_0\tilde{r} -4}{2r_0\tilde{r}}.
\end{align*}
As $-1\leq s\leq 1$, to make the set $\Nd$ non-empty, it in particular requires that
 $$\tilde{r}\geq \frac{t_0^2-r_0^2+4}{2(t_0+r_0) }= u_0 +v_0^{-1}.$$
Under our coordinates system, $\tilde{r}\leq t_0$. We discuss several cases depending on the range of $s$ and $\tilde{r}$. 

For the case when $\tilde{r}\geq v_0-0.1 u_0$, it can be showed that
\begin{align*}
s_*\geq 1,\quad r\tau &=\tilde{r}-r_0s\geq 0,\quad r\geq \frac{1}{2}(\tilde{r}-r_0s+\sqrt{1-s^2}r_0)\geq \frac{1}{10}(u_0+\sqrt{1-s}r_0),\\
u_+v_+ &\sim 1+|r^2-t^2|=1+|r_0^2-t_0^2+2\tilde{r} (t_0-r_0s)|.
\end{align*}
 Therefore we can estimate that
\begin{align*}
 & \int_{\Nd \cap\{\tilde{r}\geq v_0-0.1 u_0\}}((1+\tau)v_+^{\ga}+u_+^{\ga})^{-l} v_+^{-\a+\b}(v_+u_+)^{-\b}\tilde{r}^{k} d\tilde{r}d\tilde{\om}\\
 &\les t_0^{k}\int_{-1}^{1}\int_{v_0 -0. 1 u_0}^{t_0} (u_0+\sqrt{1-s}r_0 )^{-l \ga-\a+\b} (1+r_0^2-t_0^2+2\tilde{r} (t_0-r_0s))^{-\b}d\tilde{r}ds\\
  &\les t_0^{k}\int_{-1}^{1}  (u_0+\sqrt{1-s}r_0 )^{-l\ga -\a+\b} (t_0-r_0s)^{-1} (1+r_0^2-t_0^2+2t_0 (t_0-r_0s))^{1-\b} ds\\
  &\les t_0^{k}\int_{-1}^{1}  (u_0+\sqrt{1-s}r_0 )^{-l\ga -\a+2-\b} (u_0+r_0(1-s))^{-1}  ds .
\end{align*}
For the case when $r_0\leq \f12 t_0$, we trivially have that
\begin{align*}
\int_{-1}^{1}  (u_0+\sqrt{1-s}r_0 )^{-l\ga -\a+2-\b} (u_0+r_0(1-s))^{-1}  ds\les t_0^{-l\ga-\a-\b+1}.
\end{align*}
as $u_0=\f12(t_0-r_0)\geq \frac{1}{4}t_0$. When $r_0\geq \f12 t_0$, we show that
\begin{equation}
\label{eq:It:00}
\begin{split}
&\int_{-1}^{1}  (u_0+\sqrt{1-s}r_0 )^{-l\ga -\a+2-\b} (u_0+r_0(1-s))^{-1}  ds\\
 &\les  \int_{-1}^{1-r_0^{-2}u_0^2} (r_0(1-s))^{-1}(\sqrt{1-s}r_0)^{-l\ga-\a-\b+2}ds+u_0^{-1}\int_{1-r_0^{-2}u_0^2}^{1} u_0^{-l\ga-\a-\b+2}ds\\
&\les r_0^{-2}u_0^{3-l\ga-\a-\b}+ r_0^{-l\ga-\a-\b+1}  \Big[1+(r_0^{-2}u_0^2)^{1-\frac{l\ga+\a+\b}{2}}\Big]\\
&\les t_0^{-2}u_0^{3-l\ga-\a-\b}+t_0^{1-l\ga-\a-\b}.
\end{split}
\end{equation}
We thus conclude that
\begin{align*}
 \int_{\Nd \cap\{\tilde{r}\geq  v_0-0.1 u_0\}}((1+\tau)v_+^{\ga}+u_+^{\ga})^{-l} v_+^{-\a} u_+^{-\b}\tilde{r}^{k} d\tilde{r}d\tilde{\om}\les t_0^{k-2} (u_0^{3-l\ga-\a-\b}+t_0^{3-l\ga-\a-\b}).
\end{align*}
Next we estimate the integral on the region $u_0\leq \tilde{r}\leq v_0-0.1 u_0$. 
For the case when 
\begin{align*}
  r_0^2\leq 4+\frac{1}{2}t_0^2,
\end{align*}
in particular we always have 
\begin{align*}
s_*\leq r_0^{-1}\tilde{r}.
\end{align*}
If $r_0$ is small compared to $t_0$, that is, $r_0\leq \frac{1}{10}t_0$, then
\begin{align*}
  r\geq \tilde{r}-r_0s \geq  u_0-r_0\geq \frac{t_0}{3}.
\end{align*}
Otherwise if $\frac{1}{10}t_0\leq r_0\leq \sqrt{4+\frac{t_0^2}{2}} $, then for all $s\leq s_*$ we can show that
\begin{align*}
  2r\geq \tilde{r}-r_0s+\sqrt{1-s}r_0&\geq \tilde{r}-r_0+\sqrt{\frac{(t_0-r_0)(t_0+r_0-2\tilde{r})+4}{2\tilde{r}r_0}}r_0\\
  &\geq \tilde{r}-r_0+\frac{1}{3}\sqrt{u_0(v_0-\tilde{r})+1}\\
  &\geq \frac{1}{10}t_0,\quad  u_0 \leq \tilde{r}\leq v_0.
\end{align*}
Thus for the case when $ r_0^2\leq 4+\frac{1}{2}t_0^2$, we always have
\begin{align*}
t_0\les r\les v_+,\quad s\leq s_*\leq r_0^{-1}\tilde{r},\quad t_0\les \frac{t_0-r_0}{2}-2\leq \tilde{r}.
\end{align*}
Therefore we can bound that 
\begin{align*}
 &\int_{u_0 }^{v_0-0.1 u_0} \int_{-1}^{  s_* } ((r+\tilde{r}-r_0s) r^{\ga-1}+u_+^{\ga})^{-l}v_+^{-\a+\b}(u_+v_+)^{-\b}\tilde{r}^{k} d\tilde{r}ds\\
 &\les t_0^{-l\ga-\a+k+\b}\int_{u_0 }^{v_0-0.1 u_0} \int_{-1}^{  s_* }     (1+r_0^2-t_0^2+2\tilde{r} (t_0-r_0 s))^{-\b}  d\tilde{r} ds\\
&\les t_0^{-l\ga-\a+k+\b}\int_{u_0 }^{v_0-0.1 u_0}  r_0^{-1}\tilde{r}^{-1}  (1+r_0^2-t_0^2+2\tilde{r} (t_0+r_0))^{1-\b}  d\tilde{r}\\
&\les  t_0^{-l\ga-\a+k+\b} r_0^{-1}t_0^{-1} (t_0+r_0)^{-1} (1+|r_0^2-t_0^2+(v_0-0.1 u_0)(t_0+r_0)|)^{2-\b}\\
&\les  t_0^{-l\ga-\a+k+1-\b} .
\end{align*}
Here note that $ r_0^2 \leq 4+\frac{t_0^2}{2}$ and $\b<\min\{1, \a\}$. Combining with the above estimate on the region $\tilde{r}\geq v_0-0.1 u_0$, we have shown estimate \eqref{eq:NLW:def:3d:ex:ID:H} for the case when $r_0^2\leq 4+\frac{t_0^2}{2}$.

In the sequel, let's assume that $r_0^2>4+\frac{1}{2}t_0^2$ and $u_0\leq \tilde{r}\leq v_0-0.1 u_0$.
On the region when $r_0\leq \tilde{r}\leq v_0-0.1 u_0$, 
note that 
\begin{align*}
r_0\geq \frac{t_0}{2},\quad (\tilde{r}-\frac{t_0}{2})^2\geq (r_0-\frac{t_0}{2})^2\geq \frac{r_0^2}{2}-2-\frac{t_0^2}{4},\quad 1-s\geq 1-s_*\sim r_0^{-2}u_0 (v_0-\tilde{r})\sim  r_0^{-2}u_0^{2}.
\end{align*}
In particular we have 
$$s_*\leq r_0^{-1}\tilde{r},\quad t_0\les r_0,\quad r\sim \tilde{r}-r_0+\sqrt{1-s}r_0\sim u_0+\sqrt{1-s}r_0.$$ 
We thus can show that
\begin{align*}
 &\int_{r_0}^{v_0-0.1 u_0} \int_{-1}^{   s_* } ((r+\tilde{r}-r_0s) r^{\ga-1}+u_+^{\ga})^{-l}v_+^{-\a}u_+^{-\b}\tilde{r}^{k} d\tilde{r}ds\\
 &\les t_0^k\int_{r_0}^{v_0-0.1 u_0} \int_{-1}^{  s_* } ( u_0+\sqrt{1-s}r_0)^{-l\ga-\a+\b} (1+r_0^2-t_0^2+2\tilde{r}(t_0-r_0s))^{-\b}  d\tilde{r}ds  \\
 &\les t_0^{k }     \int_{-1}^{  1 } ( u_0+\sqrt{1-s}r_0)^{-l\ga-\a+\b}  (t_0-r_0s)^{-1} (u_0+\sqrt{1-s}r_0)^{2-2\b}   ds \\
  &  \les t_0^{k }   \int_{-1}^{ 1 } (u_0+r_0(1-s))^{-1} ( u_0+\sqrt{1-s}r_0)^{-l\ga-\a+2-\b} ds \\
  &\les t_0^{k-2}u_0^{3-l\ga-\a-\b}+t_0^{k+1-l\ga-\a-\b}.
\end{align*}
The last step follows from the computation in \eqref{eq:It:00}.

Next we estimate the integral on the region $u_0\leq \tilde{r}\leq  \f12 r_0$. 
 Since $ r_0^2\geq 4+\frac{1}{2}t_0^2$, $t_0\geq 20$ and $s\leq s_*$, in particular we have
 \begin{align*}
 &r\geq t\geq t_0-\tilde{r}\geq t_0-\f12 r_0\geq \frac{1}{10}t_0,\\
 & r+\tilde{r}-r_0s=\frac{(1-s^2)r_0^2}{r+r_0s-\tilde{r}}\sim (1-s)t_0,\\
 &1-s_*
 =\frac{(t_0-r_0)(t_0+r_0-2\tilde{r})+4}{2\tilde{r}r_0}\geq \frac{u_0}{10 \tilde{r}},\\
 & u_+\sim 1+t_0^{-1}(r_0 \tilde{r}(1-s)-2u_0 (v_0-\tilde{r})). 
\end{align*}
Therefore we can estimate that
\begin{align*}
 &\int_{u_0 }^{\f12 r_0} \int_{-1}^{  s_* } ((r+\tilde{r}-r_0s) r^{\ga-1}+u_+^{\ga})^{-l}v_+^{-\a}u_+^{-\b}\tilde{r}^{k} d\tilde{r}ds\\
 &\les  \int_{u_0 }^{\f12 r_0} \int_{-1}^{  1- 10 \tilde{r}^{-1}u_0 } ( (1-s)t_0^{\ga})^{-l}t_0^{-\a} u_0^{-\b}\tilde{r}^{k} d\tilde{r}ds+ \int_{u_0 }^{\f12 r_0} \int_{1- 10\tilde{r}^{-1}u_0}^{  s_* } ( \tilde{r}^{-1}u_0 t_0^{\ga})^{-l}t_0^{-\a}( u_+)^{-\b}\tilde{r}^{k} d\tilde{r}ds\\
 &\les \int_{u_0 }^{\f12 r_0}   ( 1+\tilde{r}^{l-1}u_0^{1-l})t_0^{-l\ga-\a} u_0^{-\b}\tilde{r}^{k} d\tilde{r} + \int_{u_0 }^{\f12 r_0}   ( \tilde{r}^{-1}u_0 t_0^{\ga})^{-l}t_0^{-\a}\tilde{r}^{-1} u_0^{1-\b}\tilde{r}^{k} d\tilde{r}\\
 &\les t_0^{-l\ga-\a}u_0^{-\b}(t_0^{k+1}+u_0^{k+1}+u_0^{1-l}t_0^{k+l}).
\end{align*}
Finally consider the integral on $\f12 r_0\leq \tilde{r}\leq  r_0$ with the condition $r_0^2>4+\frac{1}{2}t_0^2$. When $s\leq \min\{s_*, r_0^{-1}\tilde{r}\}$, we bound that 
\begin{align*}
 &\int_{\f12 r_0 }^{ r_0} \int_{-1}^{  \min\{s_*, r_0^{-1}\tilde{r}\} } ((r+\tilde{r}-r_0s) r^{\ga-1}+u_+^{\ga})^{-l}v_+^{-\a}u_+^{-\b} \tilde{r}^{k} d\tilde{r}ds\\
 &\les t_0^{k}\int_{\f12 r_0 }^{ r_0 } \int_{-1}^{  \min\{s_*, r_0^{-1}\tilde{r}\}}  (\sqrt{1-s}r_0)^{-l\ga-\a+\b}(1+r_0^2-t_0^2+2\tilde{r}(t_0-r_0s))^{-\b}  d\tilde{r}ds  \\
 &\les t_0^{k} \int_{-1}^{ 1-2r_0^{-2}(1+u_0^2 )\}}  \int_{ \max\{r_0s, \f12 r_0, \frac{t_0^2-r_0^2+4}{2(t_0-r_0s)}\} }^{ r_0 }  (\sqrt{1-s}r_0)^{-l\ga-\a+\b}(1+r_0^2-t_0^2+2\tilde{r}(t_0-r_0s))^{-\b}  d\tilde{r}ds  \\
 &\les t_0^{k} \int_{-1}^{ 1}    (u_0+\sqrt{1-s}r_0)^{-l\ga-\a+2-\b} (t_0-r_0s)^{-1}  ds  \\
 &\les t_0^{k-2}u_0^{3-l\ga-\a-\b}+t_0^{k+1-l\ga-\a-\b}.
 \end{align*}
 Again the last step has been shown in \eqref{eq:It:00}.

Lastly for the integral on $r_0^{-1}\tilde{r}<s\leq s_*$, $\frac{r_0}{2}\leq \tilde{r}\leq r_0$, it in particular requires that
\begin{align*}
  \tilde{r}\leq r_*=\f12 t_0+\sqrt{2+\f12 r_0^2-\frac{1}{4}t_0^2}<r_0.
\end{align*}
Moreover we have 
\begin{align*}
  r+\tilde{r}-r_0s=\frac{(1-s^2)r_0^2}{r+r_0s-\tilde{r}}\sim  \frac{(1-s)r_0^2}{ r}, \quad r\sim r_0s-\tilde{r}+\sqrt{1-s}r_0 \sim r_0-\tilde{r}+\sqrt{1-s}r_0 .
\end{align*}
We first estimate that 
\begin{align*}
 &\int_{ r_0-8 u_0 }^{ r_* } \int_{r_0^{-1}\tilde{r} }^{  s_* } ((r+\tilde{r}-r_0s) r^{\ga-1}+u_+^{\ga})^{-l}v_+^{-\a}u_+^{-\b}  \tilde{r}^k d\tilde{r}ds\\
 &\les t_0^{k} \int_{  r_0-8u_0 }^{ r_*} \int_{r_0^{-1}\tilde{r}}^{  s_* }  ((1-s)r_0^{2}r^{\ga-2})^{-l} r^{-\a+\b} (v_+u_+)^{-\b}    d\tilde{r}ds  \\
  &\les t_0^{k} \int_{  r_0-8u_0 }^{ r_*} \int_{r_0^{-1}\tilde{r}}^{  s_* }  (1-s)^{-\frac{l\ga+\a-\b}{2}}    r_0^{-l\ga-\a+\b} (1+r_0\tilde{r}(1-s)-2u_0(v_0-\tilde{r}))^{-\b}    d\tilde{r}ds\\
  &\les t_0^{k-2} \int_{  r_0-8u_0 }^{ r_*} \int_{1}^{ r_0(r_0-\tilde{r})+u_0^2 }  (\tilde{s}+u_0(v_0-\tilde{r}))^{-\frac{l\ga+\a-\b}{2}}      \tilde{s}^{-\b}    d\tilde{r}d \tilde{s}\\
  &\les t_0^{k-2} \int_{  r_0-8u_0 }^{ r_*}    (u_0(v_0-\tilde{r}))^{1-\frac{l\ga+\a+\b}{2}}      d\tilde{r}\\
  &\les t_0^{k-2}     u_0^{3-l\ga-\a-\b}.  
\end{align*}
Here note that $l\ga+\a+\b>2$ and $\b<1$. 
 
 Now it remains to estimate the integral on $\frac{r_0}{2}\leq \tilde{r}\leq r_0-8 u_0$ (in particular $r_0\geq 16 u_0$), $r_0^{-1}\tilde{r}\leq s\leq s_*$.  
Denote 
\begin{align*}
s_0=r_0^{-1}\tilde{r},\quad s_1=r_0^{-1}\tilde{r} (2-r_0^{-1}\tilde{r}). 
\end{align*}
It can be checked that  
\begin{align*}
s_1\leq s_*,\quad \forall \tilde{r}\leq r_0-8u_0. 
\end{align*}
On $s_0\leq s\leq s_1$, note that 
\begin{align*}
&r\sim r_0-\tilde{r}+\sqrt{1-s}r_0\sim \sqrt{1-s}r_0,\\
& 1+r_0\tilde{r}(1-s_1)-2u_0(v_0-\tilde{r})\geq \f12 (r_0-\tilde{r})^2-2u_0 (v_0-\tilde{r})\geq \frac{1}{8}(r_0-\tilde{r})^2.
\end{align*}
Similarly we can estimate that 
\begin{align*}
 &\int^{ r_0-8 u_0 }_{ \frac{r_0}{2}} \int_{s_0}^{  s_1 } ((r+\tilde{r}-r_0s) r^{\ga-1}+u_+^{\ga})^{-l}v_+^{-\a}u_+^{-\b}  \tilde{r}^k d\tilde{r}ds\\
  &\les t_0^{k} \int^{ r_0-8 u_0 }_{ \frac{r_0}{2}} \int_{s_0}^{  s_1 }  (1-s)^{-\frac{l\ga+\a-\b}{2}}    r_0^{-l\ga-\a+\b} (1+r_0\tilde{r}(1-s)-2u_0(v_0-\tilde{r}))^{-\b}    d\tilde{r}ds\\
  &\les t_0^{k-2} \int^{ r_0-8 u_0 }_{ \frac{r_0}{2}}  \int_{\frac{1}{8}(r_0-\tilde{r})^2}^{ r_0(r_0-\tilde{r})+u_0^2 }  (\tilde{s}+u_0(v_0-\tilde{r}))^{-\frac{l\ga+\a-\b}{2}}      \tilde{s}^{-\b}    d\tilde{r}d \tilde{s}\\
  &\les t_0^{k-2}  \int^{ r_0-8 u_0 }_{ \frac{r_0}{2}}     ( (r_0-\tilde{r})^2)^{1-\frac{l\ga+\a+\b}{2}}      d\tilde{r}\\
  &\les t_0^{k-2}     (u_0^{3-l\ga-\a-\b}+t_0^{3-l\ga-\a-\b}).  
\end{align*}
On the region $s_1\leq s\leq s_*$, we instead have 
\begin{align*}
&r\sim r_0-\tilde{r}+\sqrt{1-s}r_0\sim r_0-\tilde{r},\\
& 1+r_0\tilde{r}(1-s_1)-2u_0(v_0-\tilde{r})\leq  (r_0-\tilde{r})^2 .
\end{align*}
Therefore we can show that 
\begin{align*}
 &\int^{ r_0-8 u_0 }_{ \frac{r_0}{2}} \int_{s_1}^{  s_* } ((r+\tilde{r}-r_0s) r^{\ga-1}+u_+^{\ga})^{-l}v_+^{-\a}u_+^{-\b}  \tilde{r}^k d\tilde{r}ds\\
 &\les t_0^{k} \int^{ r_0-4 u_0 }_{ \frac{r_0}{2}} \int_{s_1}^{  s_* }  ((1-s)r_0^{2}r^{\ga-2})^{-l} r^{-\a+\b} (v_+u_+)^{-\b}    d\tilde{r}ds  \\
  &\les t_0^{k} \int^{ r_0-8 u_0 }_{ \frac{r_0}{2}} \int_{s_1}^{  s_* }  (1-s)^{-l} (r_0-\tilde{r})^{-\a+\b+l(2-\ga)}   r_0^{-2l}  (1+r_0\tilde{r}(1-s)-2u_0(v_0-\tilde{r}))^{-\b}    d\tilde{r}ds\\
  &\les t_0^{k-2} \int^{ r_0-8 u_0 }_{ \frac{r_0}{2}}  \int_{1}^{  (r_0-\tilde{r})^2 }  (r_0-\tilde{r})^{-\a+\b+l(2-\ga)}  (\tilde{s}+u_0(v_0-\tilde{r}))^{-l}      \tilde{s}^{-\b}    d\tilde{r}d \tilde{s}\\
  &\les t_0^{k-2}  \int^{ r_0-8 u_0 }_{ \frac{r_0}{2}}  (r_0-\tilde{r})^{-\a+\b+l(2-\ga)} ((u_0(r_0-\tilde{r}))^{1-l-\b}+(r_0-\tilde{r})^{2-2l-2\b})          d\tilde{r}\\
  &\les t_0^{k-2}     (u_0^{3-l\ga-\a-\b}+t_0^{3-l\ga-\a-\b}+u_0^{1-l-\b}t_0^{2-\a+l-l\ga}).  
\end{align*}
Combining all the above bounds, we have shown \eqref{eq:NLW:def:3d:ex:ID:H} and finished the proof for the Lemma.

\end{proof}

We have the following  decay estimates for $\phi^{lin}_H$ inside $\mathbf{D}$.
\begin{Prop}
\label{prop:NLW:def:3d:ex:ID:H}
Let $p$ and $\ga_0$ verify the same assumptions as in the main theorem \ref{thm:main}. Then inside the hyperboloid $\mathbf{D}$, if $p>\frac{1+\sqrt{17}}{2}$, then we have
\begin{equation}
\label{eq:phi:pt:Br:largep:lin}
|\phi^{lin}_H(t_0, x_0)|\leq C (2+t_0+|x_0|)^{-1}(2+||x_0|-t_0|)^{-\frac{\ga_0-1}{2}} \sqrt{\mathcal{E}_{1, \ga_0} }.
\end{equation}
Otherwise if $2<p\leq \frac{1+\sqrt{17}}{2}$ and $1<\ga_0<p-1$, then
\begin{equation}
\label{eq:phi:pt:Br:smallp:lin}
|\phi^{lin}_H (t_0, x_0)|\leq C  (2+t_0+|x_0|)^{-\frac{3+(p-2)^2}{(p+1)(5-p)}\ga_0}(1+||x_0|-t_0|)^{-\frac{\ga_0}{p+1}} \sqrt{\mathcal{E}_{1, \ga_0}  }
\end{equation}
 for some constant $C$ depending on $\ga_0$, $p$ and the zeroth order weighted energy $\mathcal{E}_{0, \ga_0}  $.
\end{Prop}
\begin{proof}
 Let $q=(t_0, x_0)$. Denote 
$  N(q)=\Nd, \quad r_0=|x_0|.$  
By definition, we have
\begin{equation*}
\begin{split}
4\pi\phi^{lin}_H(t_0, x_0)&=\int_{\tilde{\om}}t_0  \phi_1(x_0+t_0\tilde{\om})d\tilde{\om}+\pa_{t_0}\big(\int_{\tilde{\om}}t_0  \phi_0(x_0+t_0\tilde{\om})d\tilde{\om}   \big)-\int_{N(q)}|\phi|^{p-1} \phi \ \tilde{r} d\tilde{r}d\tilde{\om}.
\end{split}
\end{equation*}
By using the standard energy estimate, the linear evolution part verifies the following decay estimates
\begin{align*}
  |\int_{\tilde{\om}}t_0  \phi_1(x_0+t_0\tilde{\om})d\tilde{\om}|+|\pa_{t_0}\big(\int_{\tilde{\om}}t_0  \phi_0(x_0+t_0\tilde{\om})d\tilde{\om}   \big)|\les (1+v_0)^{-1}(1+|u_0|)^{-\frac{\ga_0-1}{2}}\sqrt{\mathcal{E}_{1, \ga_0} }.
\end{align*}
We now need to control the contribution of the nonlinear part from the exterior region.
Minor modification of the above argument for deriving the decay estimates for the solutions in the exterior region  also applies to the case $|t_0-r_0|\leq 10$ (that is Lemma \ref{lem:integration:ex:ab} holds for $|t_0-r_0|\leq 10$). Alternatively by moving the origin around, the previous decay estimates in the exterior region are also valid for $q=(t_0, x_0)$ with $|t_0-r_0|\leq 10$. Since $(t_0, x_0)\in \mathbf{D}$, in the sequel, it suffices to consider the case when $t_0\geq 20$ and $t_0>|x_0|+10$.

For the case when $2<p<\frac{1+\sqrt{17}}{2}$, $1<\ga_0<p-1$, 
first in view of the weighted energy estimate of Proposition \ref{prop:EF:cone:NW:3d}, we can bound that
\begin{align*}
   |\int_{N(q) }|\phi|^{p-1} \phi \ \tilde{r} d\tilde{r}d\tilde{\om}| 
  &\leq \left(\int_{\mathcal{N}^{-}(q)}|\phi|^{p+1}((1+\tau)v_+^{\ga}+u_+^{\ga})\tilde{r}^2 d\tilde{r}d\tilde{\om}\right)^{\frac{p}{p+1}}\\
  & \quad \cdot \left(\int_{\Nd}((1+\tau)v_+^{\ga}+u_+^{\ga})^{-p}\tilde{r}^{1-p} d\tilde{r}d\tilde{\om}\right)^{\frac{1}{p+1}}\\
  &\les \left(\int_{N(q) }((1+\tau)v_+^{\ga}+u_+^{\ga})^{-p}\tilde{r}^{1-p} d\tilde{r}d\tilde{\om}\right)^{\frac{1}{p+1}}
   \end{align*}
for all $\ga<\ga_0$. By using Lemma \ref{Lem:NLW:def:3d:ex:ID:H} with $l=p$, $\ga=\ga$, $\a=\b=0$, $k=1-p$, we derive that 
\begin{align*}
\int_{\mathcal{N}^{-}(q) \backslash \mathbf{D} } ((1+\tau)v_+^{\ga}+u_+^\ga)^{-p} \tilde{r}^{1-p}d\tilde{r}d\tilde{\om}
&\les  t_0^{-p-1} (u_0^{3-p\ga}+t_0^{3-p\ga}) +t_0^{-p\ga} (t_0^{2-p}+u_0^{2-p}+u_0^{1-p}t_0)\\
&\les t_0^{-p-1} u_0^{3-p\ga}  +t_0^{1-p\ga}  u_0^{1-p}. 
\end{align*}
Choosing $\ga$ sufficiently close to $\ga_0$, we then derive that
\begin{align*}
  |\phi^{lin}_H(t_0, x_0)|\les t_0^{-1} u_0^{\frac{3-p\ga}{p+1}}  +t_0^{\frac{1-p\ga}{p+1}}  u_0^{\frac{1-p}{p+1}} \les t_0^{-\frac{(p-2)^2+3}{(p+1)(5-p)}\ga_0}u_0^{-\frac{\ga_0}{p+1}}.
\end{align*}
This proves the decay estimate \eqref{eq:phi:pt:Br:smallp:lin} for the case when $2<p\leq \frac{1+\sqrt{17}}{2}$.

For the case when $p>\frac{1+\sqrt{17}}{2}$, similarly in view of the weighted energy estimate of Proposition \ref{prop:EF:cone:NW:3d}, we first estimate that 
\begin{align*}
   |\int_{N(q) }|\phi|^{p-1} \phi \ \tilde{r} d\tilde{r}d\tilde{\om}| 
  &\leq \left(\int_{\mathcal{N}^{-}(q)}|\phi|^{p+1}((1+\tau)v_+^{\ga}+u_+^{\ga})\tilde{r}^2 d\tilde{r}d\tilde{\om}\right)^{\frac{p}{p+1}}\\
  & \quad \cdot \left(\int_{\Nd}((1+\tau)v_+^{\ga}+u_+^{\ga})^{-p}\tilde{r}^{1-p} d\tilde{r}d\tilde{\om}\right)^{\frac{1}{p+1}}\\
  &\les \left(\int_{N(q) }((1+\tau)v_+^{\ga}+u_+^{\ga})^{-\frac{p-1}{2}} |\phi|^{\frac{p+1}{2}}\tilde{r}^{\frac{3-p}{2}} d\tilde{r}d\tilde{\om}\right)^{\frac{2}{p+1}},
\end{align*}
for all $\ga<\ga_0$. Since $N(q)$ lies in the exterior region, we can control $\phi$ by using the pointwise decay estimates already showed previously in this section. More precisely we have   
\begin{align*}
&\int_{N(q) }((1+\tau)v_+^{\ga}+u_+^{\ga})^{-\frac{p-1}{2}} |\phi|^{\frac{p+1}{2}}\tilde{r}^{\frac{3-p}{2}} d\tilde{r}d\tilde{\om} \\
& \les \mathcal{E}_{1, \ga_0}^{\frac{p+1}{2}} \int_{N(q) }((1+\tau)v_+^{\ga}+u_+^{\ga})^{-\frac{p-1}{2}} v_+^{-\frac{p+1}{2}} u_+^{-\frac{(\ga_0-1)(p+1)}{4}}\tilde{r}^{\frac{3-p}{2}} d\tilde{r}d\tilde{\om}
\end{align*}
By the assumption $(p-1)(\ga_0+1)>4$, we always have
\begin{align*}
\frac{(p-1)\ga_0}{2}+\frac{p+1 }{2}>3.
\end{align*}
In particular we can choose $\ga$ sufficiently close to $\ga_0$ such that $(p-1)(\ga+1)>4$. 
 By using Lemma \ref{Lem:NLW:def:3d:ex:ID:H} with $l=\frac{p-1}{2}$, $\ga=\ga$  , $\a=\frac{p+1}{2}$, $\b<\min\{1, \frac{(p+1)(\ga_0-1)}{4} \}$, $k=\frac{3-p}{2}$, we derive that 
\begin{align*}
&\int_{\mathcal{N}^{-}(q) \backslash \mathbf{D} } ((1+\tau)v_+^{\ga}+u_+^\ga)^{-\frac{p-1}{2}}v_+^{-\frac{p+1}{2}} u_+^{-\frac{(p+1)(\ga_0-1)}{4}} \tilde{r}^{\frac{3-p}{2}}d\tilde{r}d\tilde{\om}\\
&\les  t_0^{-\frac{p+1}{2}} u_0^{3-\frac{p-1}{2}\ga-\frac{p+1}{2}-\b } +t_0^{-\frac{p-1}{2}\ga-\frac{p+1}{2}}u_0^{-\b}(t_0^{\frac{5-p}{2}}+u_0^{\frac{3-p}{2}}t_0 )\\
&\les t_0^{-\frac{p+1}{2}} u_0^{-\frac{(p+1)(\ga_0-1)}{2}}
\end{align*}
with $\b$ such that 
\begin{align*}
0\leq \b <\min\{1, \frac{(p+1)(\ga_0-1)}{4}\},\quad \b+\frac{p+1}{2}+\frac{p-1}{2}\ga_0-3\geq \frac{(p+1)(\ga_0-1)}{4}. 
\end{align*}
Here we note that 
\begin{align*}
\frac{p\ga_0}{2}\geq \max\{1, \frac{5-p}{2}\}. 
\end{align*}
 Hence we have
\begin{align*}
 |\int_{N(q) }|\phi|^{p-1} \phi \ \tilde{r} d\tilde{r}d\tilde{\om}|\les \sqrt{\mathcal{E}_{1, \ga_0}} t_0^{-1}u_0^{-\frac{\ga_0-1}{2}}
\end{align*} 
and the decay estimate \eqref{eq:phi:pt:Br:largep:lin} then follows for the case when $\frac{1+\sqrt{17}}{2}<p\leq 5$. This proves estimate \eqref{eq:phi:pt:Br:largep:lin}. We hence finished the proof for the Proposition.

\end{proof}

\section{Semilinear wave equation on a truncated backward light cone}
\label{sec:comp}

We use conformal method to study the asymptotic decay properties for the solution in the interior region, which is conformal to a truncated backward light cone. The content in this section is independent and may be of independent interest.

Let $R>1$ be a constant and $\B_R$ be the ball with radius $R$ in $\mathbb{R}^3$. Denote $\cJ^+(\B_R)$ be the future maximal Cauchy development, that is, $(t, x)\in\mathbb{R}\times \mathbb{R}^{3}$ belongs $\cJ^+(\B_R)$ if and only if $x+t\om\in \B_R$ for all $\om\in \mathbb{S}^2$.
Let $\phi$ be the solution to the following nonlinear wave equation
\begin{equation}
  \label{eq:NLW:3D:conf}
  \begin{cases}
    \Box \phi= \La^{3-p}|\phi|^{p-1}\phi,\\
    \phi(0, x)=\phi_0,\quad \pa_t\phi(0, x)=\phi_1, \quad x\in \B_R
  \end{cases}
\end{equation}
on $\mathcal{J}^{+}(\B_R)$ with $\La=((R-t)^2-|x|^2)^{-1}$.

For any fixed point $q=(t_0, x_0)\in \mathcal{J}^+(\B_R)$, recall that $\mathcal{N}^{-}(q)$ is the past null cone of the point $q$ in $\mathcal{J}^{+}(\B_R)$  and $\mathcal{J}^{-}(q)$ is the past of the point $q$, that is, the region bounded by $\mathcal{N}^{-}(q)$ and $\B_R$. As defined in Section \ref{sec:notation}, the tilde coordinates and quantities are referred to those ones under the coordinates centered at the given point $q=(t_0, x_0)$. 
For some constant $0<\ga<1$, define
\begin{align*}
  \mathcal{I}_{\ga}=  \sup\limits_{q\in \cJ^+(\B_R) }\int_{\mathcal{N}^{-}(q)} ( v_*^\ga+(1-\tau)u_*^{\ga})   \La^{3-p} |\phi|^{p+1} d\sigma, 
\end{align*}
in which 
\[
u_*=R-t+r, \quad v_*=R-t-r,\quad \tau=\frac{x\cdot (x-x_0)}{|x||x-x_0|}=\om\cdot \tilde{\om},\quad \La=u_*^{-1}v_*^{-1}. 
\]
 The associated linear evolution to \eqref{eq:NLW:3D:conf} is denoted as $\phi^{lin}$, that is,
\begin{align*}
  \Box\phi^{lin}=0,\quad \phi^{lin}(0, x)=\phi_0,\quad \pa_t\phi^{lin}(0, x)=\phi_1,\quad t+|x|\leq R.
\end{align*}

To avoid too many constants, we make a convention in this section that $A\les B$ means that there is a constant $C$, depending only on $R$, $p$, $\ga$, $ \mathcal{I}_{\ga}$ and a fixed small constant $0<\ep<10^{-2}(1-\ga)$, such that $A\leq CB$.

We first establish two integration lemmas.
\begin{Lem}
\label{lem:bd:vga}
Fix $(t_0, x_0)\in \mathcal{J}^+(\B_R)$. For all $0\leq \tilde{r}\leq t_0$, $t=t_0-\tilde{r}$ and $r=|x_0+\tilde{r}\tilde{\om}|$, we have the following uniform bound
\begin{equation*}
\int_{\S_{(t, x_0)}(\tilde{r})}(R-t-r)^{-\ga'}d\tilde{\om}\leq C (R-t)^{\ga'}(R-t_0)^{-\ga'} (R-t_0-|x_0|+\tilde{r})^{-\ga'}
\end{equation*}
for all $ \ga'<1$ and some constant $C$ depending only on $\ga'$ and $R$. Here $\S_{(t, x_0)}(\tilde{r})$ denotes the 2-sphere with radius $\tilde{r}$ centered at $(t, x_0)$.
\end{Lem}
\begin{proof}
Denote $u_0=R-t_0+r_0$ and $v_0=R-t_0-r_0$ where $r_0=|x_0|$. By the assumption that $t_0+r_0\leq R$, we in particular have that $r\leq R-t$, which implies that
\[
(R-t-r)^{-\ga'}\leq 2 \left((R-t)^2-r^2\right)^{-\ga'}(R-t)^{\ga'}.
\]
Note that $r^2=r_0^2+\tilde{r}^2+2 \tilde{r} x_0\cdot \tilde{\om}$. We can compute that
\begin{align*}
\int_{|\tilde{\om}|=1}(R-t-r)^{-\ga'}d\tilde{\om}&\leq 4\pi (R-t)^{\ga'} \int_{-1}^{1}((R-t)^2-t_0^2-\tilde{r}^2-2\tilde{r}r_0\tau)^{-\ga'}d\tau\\
&\leq 4\pi(1-\ga')^{-1} (R-t)^{\ga'} (r_0 \tilde{r})^{-1} (u_0^{1-\ga'}(v_0+2\tilde{r})^{1-\ga'}-v_0^{1-\ga'}(u_0+2\tilde{r})^{1-\ga'}).
\end{align*}
By definition, we see that
\[
u_0(v_0+2\tilde{r})-v_0(u_0+2\tilde{r})=4\tilde{r}r_0.
\]
As $\gamma'<1$, we derive that
\[
u_0^{1-\ga'} (v_0+2\tilde{r})^{1-\ga'}-v_0^{1-\ga'}(u_0+2\tilde{r})^{1-\ga'}\leq C(R, \ga') \tilde{r}r_0 u_0^{-\gamma'}(v_0+\tilde{r})^{-\gamma'}
\]
for some constant $C(R, \ga')$ depending only on $\ga'$ and $R$.
The lemma then follows as $0\leq r_0\leq R-t_0$.
\end{proof}

The above integration lemma will be used for  case when $p>\frac{1+\sqrt{17}}{2}$. The following specific lemma will be used for small $p$.
\begin{Lem}
\label{lem:bd:vga:smallp:in}
Fix $(t_0, x_0)\in \mathcal{J}^+(\B_R)$. For all $0\leq \tilde{r}\leq t_0$, $t=t_0-\tilde{r}$, $r=|x_0+\tilde{r}\tilde{\om}|$ and $0<\ga<1$, $0\leq \a< 1$, we have the following uniform bound
\begin{equation*}
\int_{\S_{(t, x_0)}(\tilde{r})} ((1-\tau)u_*^{\ga}+v_*^{\ga})^{-\a} d\tilde{\om}\leq C (R-t_0)^{-\a \ga}
\end{equation*}
for some constant $C$ depending only on $\ga$ , $\a$ and $R$.
\end{Lem}
\begin{proof}
Using the same notations from the previous Lemma, denote $s=-\om_0\cdot \tilde{\om}$.
Recall that
  \begin{align*}
    &r^2=(\tilde{r}-r_0s)^2+(1-s^2)r_0^2,\quad \tau r =(\tilde{x}+x_0)\cdot \tilde{\om}=\tilde{r}-r_0s.
  \end{align*}
We can write the integral as
\begin{align*}
  \int_{\S_{(t, x_0)}(\tilde{r})} ((1-\tau)u_*^{\ga}+v_*^{\ga})^{-\a} d\tilde{\om}=2\pi\int_{-1}^{1} ((1-r^{-1}(\tilde{r}-r_0 s))u_*^{\ga}+v_*^{\ga})^{-\a}ds.
\end{align*}
Note that $R-t\leq u_*\leq 2(R-t)$. Hence
\begin{align*}
  \int_{\S_{(t, x_0)}(\tilde{r})} ((1-\tau)u_*^{\ga}+v_*^{\ga})^{-\a} d\tilde{\om}\les \int_{-1}^{1} ((1-r^{-1}(\tilde{r}-r_0 s))(R-t)^{\ga}+v_*^{\ga})^{-\a} ds.
\end{align*}
Here and in the following of the proof, the implicit constants rely only on $R$, $\a$, $\ga$.

For the case when $r_0\leq \frac{3}{4} (R-t_0)$, it holds that $$v_*\geq R-t_0-r_0\geq \frac{1}{4} (R-t_0),$$
from which we conclude that
\begin{align*}
  \int_{\S_{(t, x_0)}(\tilde{r})} ((1-\tau)u_*^{\ga}+v_*^{\ga})^{-\a} d\tilde{\om}\les (R-t_0)^{-\a\ga}.
\end{align*}
Otherwise for the case when $r_0\geq \frac{3}{4}(R-t_0)$, and if $\tilde{r}\leq 2 r_0 $, note that
\begin{align*}
  1-\tau=1-r^{-1}(\tilde{r}-r_0s) \geq \frac{1-s^2}{100}. 
\end{align*}
Therefore we can show that
\begin{align*}
  &\int_{-1}^{1}((1-\tau)(R-t)^{\ga}+v_*^{\ga})^{-\a} ds
  \les \int_{-1}^{1} (R-t)^{-\a\ga}(1-s^2)^{-\a} ds \les (R-t_0)^{-\a\ga}.
\end{align*}
For the remaining case $\tilde{r}\geq 2r_0$, we instead have
\begin{align*}
   v_*=R-t-r\geq v_0+10^{-2}(1+s) r_0 .
\end{align*}
Therefore we derive that
\begin{align*}
  \int_{-1}^{1}((1-\tau)(R-t)^{\ga}+v_*^{\ga})^{-\a} ds
  &\les \int_{-1}^{1} (v_0+(1+s) r_0)^{-\a\ga}  ds\\
  &\les r_0^{-1}((v_0+2r_0)^{1-\a\ga}-v_0^{1-\a\ga})\\
  &\les (R-t_0)^{-\a\ga}
\end{align*}
as $0\leq \a\ga<1$. This proves the lemma.
\end{proof}

Define
\[
\a_p=\frac{3+(p-2)^2}{(p+1)(5-p)},\quad 1<p<5.
\]
Now we are ready to prove the following estimate for the solution $\phi$.
\begin{Prop}
\label{Prop:pointwise:decay:EM}
Assume $\mathcal{I}_{\ga}$ is finite. Then the solution $\phi$ to the equation \eqref{eq:NLW:3D:conf} on $\mathcal{J}^{+}(\B_{R})$ verifies the following bounds:
\begin{itemize}
\item For the case when 
\begin{equation*}
 \frac{1+\sqrt{17}}{2}<p<5,\quad  0<\ga<1,\quad (p-1)(3-\ga)>4,
\end{equation*}
we have
\begin{equation}
\label{eq:phi:pt:Br:EM:largep}
|\phi(t_0, x_0)|\leq C \sup\limits_{|x|\leq R-t_0}|\phi^{lin}(t_0, x)|.
\end{equation}
\item
For the case when
\begin{align*}
  2<p\leq \frac{1+\sqrt{17}}{2},\quad  1<\ga_0 <p-1, \quad 0\leq \b<  \frac{p-1}{p+1}\ga_0,
\end{align*}
there exists $ \ga\in (2-\ga_0, 1) $ with finite $\mathcal{I}_{\ga}$ such that 
\begin{equation}
  \label{eq:phi:pt:Br:EM:smallp}
  |\phi(t_0, x_0)|\leq C (1+\sup\limits_{t+|x|\leq R}|\phi^{lin}u_*^{1-\b} v_*^{1-\a_p \ga_0}|)u_0^{-1+\b} v_0^{-1+\a_p\ga_0}.
\end{equation}
The constant $C$ depends only on $\mathcal{I}_{\ga} $, $R$, $p$, $\ga$ and $\b$. Here $u_*=R-t+r$, $v_*=R-t-r$ and
$u_0=R-t_0$, $v_0=R-t_0-|x_0|$.  
\end{itemize}
\end{Prop}
\begin{proof}
The proof for this Proposition relies on the representation formula for linear wave equation. The nonlinearity will be controlled by using the weighted flux bound $\mathcal{I}_{\ga}$. Recall that for any $q=(t_0, x_0)\in \mathcal{J}^{+}(\B_{R})$, we have the representation formula for the solution $\phi$
\begin{equation}
\label{eq:rep4phi:comp}
\begin{split}
4\pi\phi(t_0, x_0)&=
4\pi\phi^{lin}(t_0, x_0)
-\int_{\mathcal{N}^{-}(q)}\Box \phi \ \tilde{r} d\tilde{r}d\tilde{\om}.
\end{split}
\end{equation}
We mainly need to control the nonlinear part. For the decay estimate of \eqref{eq:phi:pt:Br:EM:largep} for the larger $p$ case, by the definition of $\mathcal{I}_{\ga}$, we can estimate that
\begin{align*}
  &|\int_{\mathcal{N}^{-}(q)}\Box \phi \ \tilde{r} d\tilde{r}d\tilde{\om}|\\
  &\leq \int_{\mathcal{N}^{-}(q)}\La^{3-p}|\phi|^{p}\ \tilde{r} d\tilde{r}d\tilde{\om}\\
  &\leq \left(\int_{\mathcal{N}^{-}(q)} v_*^{\ga}\La^{3-p}|\phi|^{p+1}\ \tilde{r}^{2} d\tilde{r}d\tilde{\om}\right)^{\frac{p-1+\ep}{p+1}}\left(\int_{\mathcal{N}^{-}(q)}v_*^{-\frac{p-1+\ep}{2-\ep}\ga}\La^{3-p}|\phi|^{\frac{(p+1)(1-\ep)}{2-\ep}}\ \tilde{r}^{\frac{3-p-2\ep}{2-\ep}} d\tilde{r}d\tilde{\om}\right)^{\frac{2-\ep}{p+1}}\\ 
  &\leq \mathcal{I}_{\ga}^{\frac{p-1+\ep}{p+1}}\left(\int_{0}^{t_0} \tilde{\cM}(t)^{ 1-\ep }\int_{\S_{(t_0-\tilde{r}, x_0)(\tilde{r})}}(R-t)^{-\frac{(p+1)(\ga+1)}{2(2-\ep)}} \La^{3-p}v_*^{-\frac{p-1+\ep}{2-\ep}\ga}\tilde{r}^{\frac{3-p-2\ep}{2-\ep}}d\tilde{\om} d\tilde{r} \right)^{\frac{2-\ep}{p+1}}.
\end{align*}
Here under the coordinates $(\tilde{t}, \tilde{x})$, we have $t=t_0-\tilde{r}$ and 
\begin{align*}
\tilde{\cM}(t)&=\sup\limits_{|x|\leq R-t}|(R-t)^{\frac{1+\ga}{2}}\phi|^{\frac{p+1}{2-\ep}},  \quad \forall 0\leq t\leq R.
\end{align*}
By the assumption $(p-1)(3-\ga)>4$, $p>1$, $0<\ga<1$,  we can choose $\ep>0$ such that 
\begin{align*}
&p-3-\frac{p-1+\ep}{2-\ep}\ga=\frac{(p-1)(2-\ga)}{2}-2-\frac{(p+1)\ep}{2(2-\ep)}> -1,\\
&\frac{3-p-2\ep}{2-\ep}+p-3-\frac{p-1+\ep}{2-\ep}\ga+1=\frac{(p-1)(1-\ga)-(p+\ga)\ep}{2-\ep}>0,\\
& \frac{p-1+\ep}{2-\ep}\ga-\frac{(p+1)(1+\ga)}{2(2-\ep)}=\frac{(p-3+2\ep)\ga-p-1}{2(2-\ep)}<0,\\
& p-3 + \frac{3-p-2\ep}{2-\ep}+p-2-\frac{p-1+\ep}{2-\ep}\ga=\frac{(3-\ga)(p-1)-4-(2p-3+\ga)\ep}{2-\ep}>0,\\
&p-2+\frac{3-p-2\ep}{2-\ep}-\frac{(p+1)(1+\ga)}{2(2-\ep)}<0 ,  \qquad \textnormal{if}  \quad p<3.
\end{align*}
Since $v_*=R-t-r$, by using Lemma \ref{lem:bd:vga}, we can bound that
\begin{align*}
  \int_{\S_{(t_0-\tilde{r}, x_0)(\tilde{r})}} v_*^{p-3-\frac{p-1+\ep}{2-\ep}\ga} d\tilde{\om}\les \left((R-t)^{-1}(R-t_0)\tilde{r}\right)^{p-3-\frac{p-1+\ep}{2-\ep}\ga}.
\end{align*}
Define
\begin{align*}
\tilde{f}(t_0, \tilde{r})=  (R-t_0)^{\frac{(1+\ga)(p+1)}{2(2-\ep)}+p-3-\frac{p-1+\ep}{2-\ep}\ga} (R-t_0+\tilde{r})^{ \frac{p-1+\ep}{2-\ep}\ga-\frac{(p+1)(1+\ga)}{2(2-\ep)}} \tilde{r}^{\frac{3-p-2\ep}{2-\ep}+p-3-\frac{p-1+\ep}{2-\ep}\ga} .
\end{align*}
We then conclude that
\begin{align*}
 |\phi(t_0, x_0)|^{\frac{p+1}{2-\ep}}&\les |\phi^{lin}(t_0, x_0)|^{\frac{p+1}{2-\ep}}+ \int_{0}^{t_0}\tilde{\cM}(t_0-\tilde{r} )^{1-\ep} (R-t_0)^{-\frac{(1+\ga)(p+1)}{2(2-\ep)}} \tilde{f}(t_0, \tilde{r}) d\tilde{r},
\end{align*}
which implies that
\begin{align}
\label{eq:cM:0}
\tilde{\cM}(t_0)&\les \sup\limits_{|x_0|\leq R-t_0}|\phi^{lin}(t_0, x_0)(R-t_0)^{\frac{1+\ga}{2}}|^{\frac{p+1}{2-\ep}} + \int_{0}^{t_0}\tilde{\cM}(t_0-\tilde{r} )^{1-\ep}  \tilde{f}(t_0,\tilde{r}) d\tilde{r}.
\end{align}
We now check that $\tilde{f}(t_0, \tilde{r})$ is integrable on $[0, t_0]$ with respect to $\tilde{r}$. For the case when $p\geq 3$, by the choice of $\ep$, we can bound that 
\begin{align*}
  \int_0^{t_0}\tilde{f}(t_0, \tilde{r})d\tilde{r}\les \int_0^{t_0} (R-t_0)^{p-3}\tilde{r}^{ \frac{3-p-2\ep}{2-\ep}+p-3-\frac{p-1+\ep}{2-\ep}\ga}d\tilde{r}\les  t_0^{\frac{(p-1)(1-\ga)-(p+\ga)\ep}{2-\ep}}\les 1.
\end{align*}
Here the implicit constant relies only on $p$, $\ga$ and $R$.
For the case when $p<3$, split the integral into two parts. For the integral on $[0, \min\{R-t_0, t_0\}]$, similarly we have 
\begin{align*}
  \int_0^{ \min\{R-t_0, t_0\}} \tilde{f}(t_0, \tilde{r})d\tilde{r} &\leq \int_0^{\min\{R-t_0, t_0\} }(R-t_0)^{ p-3 }   \tilde{r}^{\frac{3-p-2\ep}{2-\ep}+p-3-\frac{p-1+\ep}{2-\ep}\ga} d\tilde{r} \\
  &\les (R-t_0)^{ p-3 + \frac{3-p-2\ep}{2-\ep}+p-2-\frac{p-1+\ep}{2-\ep}\ga}\les 1.
\end{align*}
For the remaining part with $p<3$, we have 
\begin{align*}
  \int_{\min\{R-t_0, t_0\}}^{t_0}\tilde{f}(t_0, \tilde{r}) d\tilde{r}&\leq \int_{\min\{R-t_0, t_0\}}^{t_0} (R-t_0)^{\frac{(1+\ga)(p+1)}{2(2-\ep)}+p-3-\frac{p-1+\ep}{2-\ep}\ga}  \tilde{r}^{\frac{3-p-2\ep}{2-\ep}+p-3 -\frac{(p+1)(1+\ga)}{2(2-\ep)}}   d\tilde{r}\\
  &\les   (R-t_0)^{ \frac{(1+\ga)(p+1)}{2(2-\ep)}+p-3-\frac{p-1+\ep}{2-\ep}\ga+\frac{3-p-2\ep}{2-\ep}+p-2 -\frac{(p+1)(1+\ga)}{2(2-\ep)}} \les 1.
\end{align*}
In view of \eqref{eq:cM:0}, a bootstrap argument then implies that 
\begin{align*}
  |\phi(t_0, x_0)|\les \sup\limits_{|x|\leq R-t_0} |\phi^{lin}(t_0, x)|.
\end{align*}
This shows the bound \eqref{eq:phi:pt:Br:EM:largep}.

\bigskip

Next for estimate \eqref{eq:phi:pt:Br:EM:smallp} with small power $2<p\leq \frac{1+\sqrt{17}}{2}$, we control the nonlinearity directly by using the weighted flux for large $\tilde{r}$. We split the integral into two parts: specifically for the smaller $\tilde{r}$ on $[0, t_*]$ and larger $\tilde{r}$ on $[t_*, t_0]$ with some $0<t_*\leq t_0$ depending on $R$, $p$, $t_0$ and $r_0$. 
 For the integral on $[t_*, t_0]$, by the definition of $\mathcal{I}_{\ga}$, we can show that
\begin{align*}
  &|\int_{\mathcal{N}^{-}(q)\cap\{\tilde{r}\geq t_*\}}\Box \phi \ \tilde{r} d\tilde{r}d\tilde{\om}|\\
  &\leq \left(\int_{\mathcal{N}^{-}(q)} ((1-\tau)u_*^{\ga}+v_*)^{\ga})\La^{3-p}|\phi|^{p+1}\ \tilde{r}^{2} d\tilde{r}d\tilde{\om}\right)^{\frac{p}{p+1}}\\
  &\quad \cdot \left(\int_{\mathcal{N}^{-}(q)\cap\{\tilde{r}\geq t_*\}}((1-\tau)u_*^{\ga}+v_*)^{\ga})^{-p}\La^{3-p} \tilde{r}^{1-p} d\tilde{r}d\tilde{\om}\right)^{\frac{1}{p+1}}\\
  &\les \left(\int_{t_*}^{t_0} \int_{\S_{(t_0-\tilde{r}, x_0)}(\tilde{r})}v_*^{p-3} ((1-\tau)u_*^{\ga}+v_*)^{\ga})^{-p}  u_*^{p-3} \tilde{r}^{1-p} d\tilde{\om} d\tilde{r}\right)^{\frac{1}{p+1}}.
\end{align*}
 Since $v_*\geq  v_0$, by using Lemma \ref{lem:bd:vga:smallp:in}, we then can show that
\begin{align*}
  &|\int_{\mathcal{N}^{-}(q)\cap\{\tilde{r}\geq t_*\}}\Box \phi \ \tilde{r} d\tilde{r}d\tilde{\om}|^{p+1}\\
  &\les v_0^{p-3-(p-1)\ga-\ep}\int_{t_*}^{t_0}(R-t)^{p-3}\tilde{r}^{1-p}  \int_{\S_{(t_0-\tilde{r}, x_0)}(\tilde{r})}((1-\tau)u_*^{\ga}+v_*)^{\ga})^{-1+\ep \ga^{-1}} d\tilde{\om} d\tilde{r}\\
  &\les v_0^{p-3-(p-1)\ga-\ep} (R-t_0)^{-\ga+\ep} \int_{t_*}^{t_0}(R-t_0)^{p-3}\tilde{r}^{1-p}  d\tilde{r}\\
  &\les v_0^{p-3-(p-1)\ga-\ep} (R-t_0)^{p-3-\ga+\ep} t_*^{2-p}
\end{align*}
for all $\ep>0$. Here we recall that $p>2$. Choose $\ep$ and $t_*$ such that 
\begin{align}
\label{eq:assumption:t1}
v_0^{p-3-(p-1)\ga-\ep} (R-t_0)^{p-3-\ga+\ep} t_*^{2-p}\les (v_0^{-1+\a_p\ga_0}u_0^{-1+\b})^{p+1}.
\end{align}
Here recall that $0\leq \b< \frac{p-1}{p+1}$. 
We therefore can derive that
\begin{align*}
  |\phi(t_0, x_0)|&\les |\phi^{lin}(t_0, x_0)|+|\int_{\mathcal{N}^{-}(q) }\Box \phi \ \tilde{r} d\tilde{r}d\tilde{\om}| \\
  &\les  |\phi^{lin}(t_0, x_0)|+u_0^{-1+\b } v_0^{-1+\a_p \ga_0}+|\int_{\mathcal{N}^{-}(q)\cap\{\tilde{r}\leq t_*\}}\Box \phi  \tilde{r} d\tilde{r}d\tilde{\om}|.
\end{align*}
Similar to the large $p$ case, we use bootstrap argument to control the integral on $[0, t^*]$. First we have
\begin{align*}
  &|\int_{\mathcal{N}^{-}(q)\cap\{\tilde{r}\leq t_*\}}\Box \phi \ \tilde{r} d\tilde{r}d\tilde{\om}|\\
  &\les  \left(\int_{\mathcal{N}^{-}(q) \cap\{\tilde{r}\leq t_*\}} (v_{*}^{\ga}+(1-\tau)u_{*}^{\ga})^{-\frac{p-1+\ep}{2-\ep}}\La^{3-p}|\phi|^{\frac{(p+1)(1-\ep)}{2-\ep}}\ \tilde{r}^{\frac{3-p-2\ep}{2-\ep}} d\tilde{r}d\tilde{\om}\right)^{\frac{2-\ep}{p+1}} 
\end{align*}
for all $\ep\geq 0$. Now assume that for some constant $M$, the following bootstrap assumption
\begin{align*}
  |\phi(t, x)|u_*^{1-\b}v_*^{ 1-\a_p \ga_0}\leq 2M ,\quad t+|x|\leq R. 
\end{align*}
holds. 
As $\ga_0<p-1$ and $p<3$, it can be checked that
\begin{align*}
  \b<\frac{p-1}{p+1}\ga_0  \leq 1,\quad 1-\a_p\ga_0-\frac{3-p}{(5-p)}>0.
\end{align*}
Notice that $v_*\geq v_0$, $u_*\geq u_0$.
By using Lemma \ref{lem:bd:vga:smallp:in} and the above bootstrap assumption, we can estimate  
\begin{align*}
& \int_{\mathcal{N}^{-}(q) \cap\{\tilde{r}\leq t_*\}} (v_{*}^{\ga}+(1-\tau)u_{*}^{\ga})^{-\frac{p-1+\ep}{2-\ep}}\La^{3-p}|\phi|^{\frac{(p+1)(1-\ep)}{2-\ep}}\ \tilde{r}^{\frac{3-p-2\ep}{2-\ep}} d\tilde{r}d\tilde{\om}\\
  & \les \int_0^{t_*} \int_{\S_{(t_0-\tilde{r}, x_0)}(\tilde{r})}((1-\tau)u_*^{\ga}+v_*^\ga)^{-\frac{p-1+\ep}{2-\ep}}\La^{3-p} (2M u_*^{-1+\b}v_*^{ -1+\a_p \ga_0})^{\frac{(p+1)(1-\ep)}{2-\ep}} \tilde{r}^{\frac{3-p-2\ep}{2-\ep}}  d\tilde{\om} d\tilde{r}\\
   & \les \int_0^{t_*}  (R-t_0)^{-\frac{p-1+\ep}{2-\ep}\ga}u_0^{p-3}v_0^{p-3}(2M u_0^{-1+\b}v_0^{ -1+\a_p \ga_0})^{\frac{(p+1)(1-\ep)}{2-\ep}} \tilde{r}^{\frac{3-p-2\ep}{2-\ep}}    d\tilde{r}\\
  &\les   (R-t_0)^{-\frac{p-1+\ep}{2-\ep}\ga}u_0^{p-3}v_0^{p-3}(2M u_0^{-1+\b}v_0^{ -1+\a_p \ga_0})^{\frac{(p+1)(1-\ep)}{2-\ep}} t_{*}^{\frac{5-p-3\ep}{2-\ep}} \\
  &\les (M^{1-\ep} u_0^{-1+\b}v_0^{ -1+\a_p \ga_0})^{\frac{(p+1) }{2-\ep}}   
\end{align*}
if we choose $t_*$ to be  
\begin{align*}
 (R-t_0)^{-\frac{p-1+\ep}{2-\ep}\ga}u_0^{p-3}v_0^{p-3}(  u_0^{1-\b}v_0^{ 1-\a_p \ga_0})^{\frac{(p+1)\ep}{2-\ep}} t_{*}^{\frac{5-p-3\ep}{2-\ep}} =1,
\end{align*}
that is, 
\begin{align*}
   t_{*} =u_0^{\frac{2-\ep}{5-p-3\ep}\left(\frac{p-1+\ep}{2-\ep}\ga+3-p-(1-\b)\frac{p+1}{2-\ep}\ep\right)} v_0^{\frac{2-\ep}{5-p-3\ep}\left( 3-p-(1-\a_p \ga_0 )\frac{p+1}{2-\ep}\ep\right)}. 
\end{align*}
We see that if $\ep$ is sufficiently small, depending only on $p$, $\ga_0$, then $t_*$ is a product of  positive powers of $u_0$ and $v_0$.
Now we checked that such $t_*$ verifies the inequality \eqref{eq:assumption:t1}.
Dropping the error terms containing $\ep$, it suffices to show that 
\begin{align*}
v_0^{p-3-(p-1)\ga +(1-\a_p \ga_0)(p+1)-\frac{2(p-2)(3-p)}{5-p}-C_1\ep} u_0^{p-3-\ga +(1-\b)(p+1)-\frac{2(p-2)(\frac{p-1}{2}\ga+3-p)}{5-p}-C_2\ep }\les 1
\end{align*}
for sufficiently small $\ep>0$. Here $C_1$, $C_2$ are constants depending only on $p$. Note that $\ga>2-\ga_0$. The above inequality requires that 
\begin{align*}
v_0^{\frac{2(p-2)(3-p)}{5-p}(\ga_0-1)-C_1\ep -(p-1)(\ga+\ga_0-2)} u_0^{-\frac{2(p-2)(3-p)}{5-p}(\ga_0-1)-C_2 \ep +(\frac{p-1}{p+1}\ga_0-\b)(p+1)-2(\ga+\ga_0-2)}\les 1.
\end{align*}
Recall that 
\[
\b<\frac{p-1}{p+1}\ga_0, \quad \ga>2-\ga_0,\quad  2<p<3, \quad \ga_0>1,\quad  v_0\leq u_0 \leq R.
\] 
The previous inequality is valid if $\ep$ and $\ga+\ga_0-2$ are sufficiently small. 

Now combining the above estimates, we derive that 
\begin{align*}
   |\phi(t_0, x_0)| \les    |\phi^{lin}(t_0, x_0)| +u_0^{-1+\b } v_0^{-1+\a_p \ga_0}+ M^{1-\ep} u_0^{-1+\b}v_0^{ -1+\a_p \ga_0} .
\end{align*}
This implies that 
\begin{align*}
  u_0^{1-\b } v_0^{1-\a_p \ga_0}|\phi(t_0, x_0)| \les    \sup\limits_{t+|x|\leq R} (u_{*}^{1-\b } v_{*}^{1-\a_p \ga_0}|\phi^{lin}(t_0, x_0)|) +1+ M^{1-\ep}. 
\end{align*}
For sufficiently large $M$, we conclude from the above inequality that 
\begin{align*}
u_0^{1-\b } v_0^{1-\a_p \ga_0}|\phi(t_0, x_0)| \leq M,\quad \forall t_0+|x_0|\leq R. 
\end{align*}
We hence improved the bootstrap assumption and the above argument implies that 
\begin{align*}
u_0^{1-\b } v_0^{1-\a_p \ga_0}|\phi(t_0, x_0)| \les \sup\limits_{t+|x|\leq R}(u_{*}^{1-\b } v_{*}^{1-\a_p \ga_0}|\phi^{lin}(t_0, x_0)|) +1 . 
\end{align*}
This finishes the proof for the Proposition.

\end{proof}

\section{The solution in the interior region and proof for the main theorem \ref{thm:main}}
The aim of this section is to apply the result of Proposition \ref{Prop:pointwise:decay:EM} from the previous section together with estimates of Proposition  \ref{prop:NLW:def:3d:ex:ID:H} to show the asymptotic decay properties for solutions to the nonlinear wave equations \eqref{eq:NLW:semi:3d} in the interior region $\{t+2\geq |x|\}$ which is contained inside the region $\mathbf{D}$, enclosed by the forward hyperboloid $\mathbb{H}$ defined in \eqref{eq:def4Hyperboloid}.

Define the conformal map
\begin{align*}
\mathbf{\Phi}:(t, x)\longmapsto (\tilde{t}, \tilde{x})=\left(-\frac{t^*}{(t^*)^2-|x|^2}+R^*,\quad \frac{x}{(t^*)^2-|x|^2}\right)
\end{align*}
from the region $\mathbf{D}$ to Minkowski space.
The image of $\mathbf{\Phi}(\mathbf{D})$ is a truncated backward light cone
\begin{align*}
\mathbf{\Phi}(\mathbf{D})=\left\{(\tilde{t}, \tilde{x})|\quad \tilde{t}+|\tilde{x}|<R^*,\quad \tilde{t}\geq 0\right\}.
\end{align*}
Denote
\begin{equation*}
\Lambda(t, x)=(t^*)^2-|x|^2.
\end{equation*}
Direct computation shows that $\tilde{\phi}=(\Lambda \phi)\circ \mathbf{\Phi}^{-1}$ (as a scalar field in $(\tilde{t}, \tilde{x})$ variables on $\mathbf{\Phi}(\mathbf{D})$) verifies the nonlinear wave equation \eqref{eq:NLW:3D:conf}. For simplicity we may identify $\Lambda \phi$ with $(\Lambda \phi)\circ \mathbf{\Phi}^{-1}$.
The initial hypersurface for the above backward light cone is a ball with radius $R^*$
$$ \mathbf{\Phi}(\mathbb{H}) =\{(0, \tilde{x})||\tilde{x}|\leq R^*\}.$$
By doing this conformal transformation, the Cauchy problem of equation \eqref{eq:NLW:semi:3d} with initial hypersurface $\Hy$ is then equivalent to the Cauchy problem of equation \eqref{eq:NLW:3D:conf} with initial hypersurface $\mathbf{\Phi}(\mathbb{H}) $.

\def\PD{\mathbf{\Phi}(\mathbf{D})}
To apply the result of Proposition \ref{Prop:pointwise:decay:EM}, we need to control the weighted energy flux $\mathcal{I}_{\ga}$ and the linear evolution $\tilde{\phi}^{lin}$ on $\mathbf{\Phi}(\mathbf{D})$. Let's first show that $\mathcal{I}_{\ga}$ is uniformly bounded for all $2-\ga_0<\ga<1$. 
\begin{Prop}
  \label{prop:bd:ID:WF}
  Let $\ga\in (2-\ga_0, 1)$. The solution $\tilde{\phi}= \La\phi$ on $\mathbf{\Phi}(\mathbf{D})$ verifies the following bounds
  \begin{align}
    \label{eq:bd:IDSB:WF}
  \sup\limits_{\tilde{q}\in \cJ^+(\PD) }\int_{\mathcal{N}^{-}(\tilde{q})} ( v_*^\ga+(1-\tilde{\tau})u_*^{\ga})   \La^{3-p} |\tilde{\phi}|^{p+1} d\tilde{\sigma} 
  \leq C\mathcal{E}_{0, \ga_0} 
  \end{align}
  for some constant C depending on $\ga_0$, $p$ and $\ep$. 
\end{Prop}
\begin{proof}
This bound essentially follows from Proposition \ref{prop:EF:cone:NW:3d} since the conformal map $\mathbf{\Phi}$ preserves the null curves. Let's fix a point $q=(t_0, x_0)$ in $\mathbf{D}$ with image $\tilde{q}=(\tilde{t}_0, \tilde{x}_0)$ in $\PD$. We use Cartesian coordinates $(s, y)$ or polar coordinates $(s, l, \theta)$ centered at $q$. The tilde ones are the associated symbols on $\PD$. We need to find out the transformation of the surface measure under the conformal map $\mathbf{\Phi}$. By definition, we have 
\begin{align*}
&\tilde{t}=R^{*}-\La^{-1}t^{*}, \quad \tilde{x}= \La^{-1}x, \quad \tilde{\tau}=\tilde{\om}\cdot \frac{\tilde{x}-\tilde{x}_0}{|\tilde{x}-\tilde{x}_0|}=\om\cdot \frac{\tilde{x}-\tilde{x}_0}{|\tilde{x}-\tilde{x}_0|},\quad \tau=\om\cdot\frac{x-x_0}{|x-x_0|}.
\end{align*} 
On the null hypersurface $\mathcal{N}^{-}(q)$, one has $t-t_0+|x-x_0|=0$. Similarly on $\mathcal{N}^{-}(\tilde{q})$, we have 
\begin{align*}
\tilde{t}^*-\tilde{t}_0^*+|\tilde{x}-\tilde{x}_0|=0. 
\end{align*}
In particular we have the surface measure
\begin{align*}
& d\tilde{\sigma}= \sqrt{2}d\tilde{x},\quad  d \sigma= \sqrt{2}dx.  
\end{align*}
By definition and making use of the relation $t-t_0+|x-x_0|=0$, we compute that 
\begin{align*}
d\tilde{x}_i&= \La^{-1} dx_i -\La^{-2} x_i (2t^* dt^*-2\sum\limits_{j=1}^{3}x_j d x_j)\\
&=\La^{-1} dx_i+\La^{-2} x_i ( 2t_0^* d|x-x_0|+2 x_{0j} dx_j)\\
&= \La^{-1}dx_i +2\La^{-2}x_i ( t_0^{*}|x-x_0|^{-1} (x-x_0)+x_0)dx.
\end{align*}
Therefore we derive that 
\begin{align*}
d\tilde{x}&= \La^{-3} dx +2\La^{-4}  (t_0^{*}|x-x_0|^{-1} x\cdot (x-x_0)+x\cdot x_0)dx\\
&= \La^{-4} (\La +2 t_0^{*}|x-x_0|^{-1} x\cdot (x-x_0)+2 x\cdot x_0)dx.
\end{align*}
From the relation $t-t_0+|x-x_0|=0$, we conclude that 
\begin{align*}
2x\cdot x_0= -(t^*-t_0^*)^2 +r^2+r_0^2.
\end{align*}
Here recall that $t^*=t+3$. Hence we compute that 
\begin{align*}
\La +2 t_0^{*}|x-x_0|^{-1} x\cdot (x-x_0)+2 x\cdot x_0&=\La +2 t_0^{*}|x-x_0|^{-1}r^2+(-(t^*-t_0^*)^2 +r^2+r_0^2)(1-s^{-1}t_0^*) \\
&=s^{-1}(t^* \La_0-\La  t_0^{*}),\quad s=|x-x_0|=t_0-t.
\end{align*}
This leads to the transformation of the surface measure 
\begin{align*}
d\tilde{x}= \La^{-4} s^{-1} | t^* \La_0-\La  t_0^{*}| dx. 
\end{align*}
Next we compute $\tilde{\tau}$. By definition, we can show that 
\begin{align*}
\tilde{\tau}&=\om\cdot \frac{\tilde{x}-\tilde{x}_0}{|\tilde{x}-\tilde{x}_0|}\\
&=(\tilde{t}_0^*-\tilde{t}^*)^{-1} (\La^{-1} r- \La_0^{-1} \om\cdot x_0)\\
&= (\La^{-1}t^* -\La_0^{-1}t_0^*)^{-1} (\La^{-1} r- \La_0^{-1} \frac{r^2+r_0^2-s^2}{2r})\\
&= (2r)^{-1}(\La_0 t^*-\La t_0^*)^{-1}(2\La_0 r^2-\La (r^2+r_0^2-s^2)).
\end{align*}
Here we denote 
\begin{align*}
  v=t^*+r,\quad u= t^*-r. 
\end{align*}
Therefore we can compute that 
\begin{align*}
(1-\tilde{\tau})u_*^{\ga}&= (\La^{-1} (t^*+r))^{\ga}  (2r)^{-1}(\La_0 t^*-\La t_0^*)^{-1} (2r\La_0(t^*-r)+\La(r^2+r_0^2-s^2-2rt_0^*))\\
&=  (2r)^{-1}(\La_0 t^*-\La t_0^*)^{-1} u^{2-\ga} (v-u_0)(v_0-v).  
\end{align*}
On the other hand, note that 
\begin{align*}
1-\tau =1- \frac{r-x_0\cdot \om}{s}=\frac{2rs-2r^2+r_0^2+r^2-s^2}{2rs}=\frac{(r_0+r-s)(r_0-r+s)}{2rs}=\frac{(v-u_0)(v_0-v)}{2rs}.
\end{align*}
This shows that 
\begin{align*}
(1-\tilde{\tau})u_*^{\ga}\La^{3-p}|\tilde{\phi}|^{p+1} d\tilde{\sigma}&= (2r)^{-1}(\La_0 t^*-\La t_0^*)^{-1} u^{2-\ga} 2rs(1-\tau) \La^{3-p}\La^{p+1}|\phi|^{p+1} \La^{-4} s^{-1} | t^* \La_0-\La  t_0^{*}| d\sigma\\
&=(1-\tau)u^{2-\ga} |\phi|^{p+1}d\sigma.
\end{align*}
Similarly we have 
\begin{align*}
(1+\tilde{\tau})v_*^{\ga}\La^{3-p}|\tilde{\phi}|^{p+1} d\tilde{\sigma} 
=(1+\tau)v^{2-\ga} |\phi|^{p+1}d\sigma. 
\end{align*}
Since $0<u\leq v$, $v_*\leq u_*$, $|\tilde{\tau}|\leq 1$, in view of the weighted energy estimate of Proposition \ref{prop:EF:cone:NW:3d}, we thus conclude that 
 \begin{align*}
    \int_{\mathcal{N}^{-}(\tilde{q})} ( v_*^\ga+(1-\tilde{\tau})u_*^{\ga})   \La^{3-p} |\tilde{\phi}|^{p+1} d\tilde{\sigma}  \leq 2  \int_{\mathcal{N}^{-}( q) \cap \mathbf{D}} ( u^{2-\ga}+(1+\tau)v^{2-\ga})     | \phi|^{p+1} d\sigma
  \les \mathcal{E}_{0, \ga_0} 
  \end{align*}
  for all $1>\ga>2-\ga_0$. In particular estimate \eqref{eq:bd:IDSB:WF} holds and we finished the proof for the proposition. 

  \end{proof}

Now we can finish the proof for the main Theorem \ref{thm:main}  by showing the pointwise decay estimates for the solution $\phi$ to \eqref{eq:NLW:semi:3d} in the interior region.  As indicated previously,  $\tilde{\phi}=\La\phi$ solves equation \eqref{eq:NLW:3D:conf} on the compact region $\mathbf{\Phi}(\mathbf{D})$ for solution $\phi$ to \eqref{eq:NLW:semi:3d}. The above Proposition shows that $\mathcal{I}_{\ga}$ is finite for all $2-\ga_0<\ga<1$. Hence we can apply Proposition \ref{Prop:pointwise:decay:EM}.
Setting
\begin{equation*}
 \ga=2-\ga_0+\ep
\end{equation*}
with 
\begin{equation*}
0<\ep<10^{-1}\min\{\ga_0-1, 2-\ga_0, |\ga_0+1-\frac{4}{p-1}|\}.
\end{equation*}
By our assumption on $\ga_0$, we in particular have $0<\ga<1$.

 For the case when
$$\frac{1+\sqrt{17}}{2}<p<5, \quad \max\{\frac{4}{p-1}-1, 1\}<\ga_0<\min\{p-1, 2\},$$
the choice of $\ep$ also implies that
\begin{align*}
  (p-1)(3-\ga)=(p-1)(1+\ga_0-\ep)>4.
\end{align*}
Then from Proposition \ref{Prop:pointwise:decay:EM}, we conclude that
\begin{align*}
  |\La\phi|(\tilde{t}, \tilde{x})&\les  \sup\limits_{|\tilde{y}|\leq \tilde{t}}|\tilde{\phi}^{lin}(\tilde{t}, \tilde{y})|.
\end{align*}
Here $\tilde{\phi}^{lin}$ is the linear evolution with initial data $(\tilde{\phi}(0, \tilde{x}), \tilde{\pa}_{\tilde{t}}\tilde{\phi}(0, \tilde{x}))$.
By the conformal transformation, $\tilde{\phi}^{lin}$ can be identified with $\La \phi_H^{lin}$, in which $\phi_H^{lin}$ was defined before Proposition \ref{prop:NLW:def:3d:ex:ID:H}.
Recall that
\[
\tilde{t}=R^*-\La^{-1}(t+3),\quad |\tilde{x}|=\La^{-1}r, \quad \La =(t+3-r)(t+3+r).
\]
And inside the hyperboloid $\mathbb{H}$, we have $v_+\leq t+3$. Thus
\begin{align*}
 \frac{1}{8}u_+^{-1}&\leq  R^*-\tilde{t}=\La^{-1}(t+3)\leq u_+^{-1},\\
 \frac{1}{4}v_+^{-1}&\leq  R^*-\tilde{t}-|\tilde{x}|=\La^{-1}(t+3-r)=(t+3+r)^{-1}\leq  v_+^{-1}.
\end{align*}
Therefore by using the decay estimate \eqref{eq:phi:pt:Br:largep:lin} of Proposition \ref{prop:NLW:def:3d:ex:ID:H}, we then conclude that
\begin{align*}
|\tilde{\phi}^{lin}(\tilde{t}, \tilde{x})|&\les \La |\phi^{lin}_H|\les \sqrt{\mathcal{E}_{1, \ga_0} } (2+t+|x|)^{-1}(2+||x|-t|)^{-\frac{\ga_0-1}{2}}\La\\
&\les \sqrt{\mathcal{E}_{1, \ga_0} }  (R^*-\tilde{t})^{-\frac{1+\ga_0}{2}},
\end{align*}
which leads to
\begin{align*}
  |\phi(t, x)|\les \sqrt{\mathcal{E}_{1, \ga_0} }  \La^{-1}(R^*-\tilde{t})^{-\frac{1+\ga_0}{2}} \les \sqrt{\mathcal{E}_{1, \ga_0} } v_+^{-1}u_+^{-\frac{\ga_0-1}{2}}.
\end{align*}
This proves the pointwise decay estimate for the solution in the interior region for the large $p$ case.

Finally for the small $p$ case, take
\[
\b=\frac{\ga_0}{p+1}.
\]
The small positive constant $\ep$ can be chosen so that $\b\leq \frac{p-1}{p+1}\ga_0-1$ as $\ga_0>1$.
Then by using the linear decay estimate \eqref{eq:phi:pt:Br:smallp:lin} of Proposition \ref{prop:NLW:def:3d:ex:ID:H}, we can show that
\begin{align*}
  |\tilde{\phi}^{lin}u_*^{1-\b} v_*^{1-\a_p \ga_0}|&\les |\La \phi_{H}^{lin} u_+^{\b-1} v_+^{-1+\a_p \ga_0}|\\
  &\les \sqrt{\mathcal{E}_{1, \ga_0}  } u_+ v_+ v_+^{-\a_p \ga_0}u_+^{-\b}u_+^{\b-1} v_+^{-1+\a_p \ga_0}\\
  &\les \sqrt{\mathcal{E}_{1, \ga_0}  }.
\end{align*}
Here recall that $u_*=R^*-\tilde{t}+|\tilde{x}|$, $v_*=R^*-\tilde{t}-|\tilde{x}|$. Hence from estimate \eqref{eq:phi:pt:Br:EM:smallp} of Proposition \ref{Prop:pointwise:decay:EM}, we conclude that
\begin{align*}
  |\tilde{\phi}(\tilde{t}_0, \tilde{x}_0)|&\les (1+\sup\limits_{\tilde{t}+|\tilde{x}|\leq R*}|\tilde{\phi}^{lin}u_*^{1-\b} v_*^{1-\a_p \ga_0}|)(R^*-\tilde{t}_0)^{-1+\b} (R^*-\tilde{t}_0-|\tilde{x}_0|)^{-1+\a_p\ga_0}\\
  &\les (1+\sqrt{\mathcal{E}_{1, \ga_0} }) (R^*-\tilde{t}_0)^{-1+\b} (R^*-\tilde{t}_0-|\tilde{x}_0|)^{-1+\a_p\ga_0},
\end{align*}
which implies that in the interior region for the case when $2<p\leq \frac{1+\sqrt{17}}{2}$, the solution $\phi$ of \eqref{eq:NLW:semi:3d} verifies the following decay estimate
\begin{align*}
  |\phi(t, x)|\leq |\La^{-1}\tilde{\phi}(\tilde{t}, \tilde{x})|\les (1+\sqrt{\mathcal{E}_{1, \ga_0} }) u_+^{-1}v_+^{-1}u_+^{1-\b} v_+^{1-\a_p\ga_0}\les (1+\sqrt{\mathcal{E}_{1, \ga_0} }) u_+^{-\b} v_+^{-\a_p\ga_0}.
\end{align*}
Here the implicit constant replies on $\mathcal{E}_{0, \ga_0} $, $\ga_0$ and $p$.
We thus complete the proof for the main Theorem \ref{thm:main}.

\bigskip

As for the scattering result of Corollary \ref{cor:scattering:3D}, by using the standard energy estimate, the solution scatters in energy space if the mixed norm $\|\phi\|_{L_t^p L_x^{2p}}$ of the solution is finite (see e.g. Lemma 4.4 in \cite{Strauss78:NLW}). Moreover, it has been shown in the author's companion paper, the solution scatters in $\dot{H}^{s}$ for all $\frac{3}{2}-\frac{2}{p-1}\leq s\leq 1$ for the case when $p>\frac{1+\sqrt{17}}{2}$. In particular, it suffices to consider the small $p$ case when $2<p\leq \frac{1+\sqrt{17}}{2}$. By using the pointwise decay estimate of the main theorem \ref{thm:main}, we estimate that
\begin{align*}
  \|\phi\|_{L_t^p L_x^{2p}}^p&=\int_{\mathbb{R}} \left(\int_{\mathbb{R}^3} |\phi|^{p+1}v_+^{\ga_0-1-\ep}|\phi|^{p-1}v_+^{-\ga_0+1+\ep}dx\right)^{\f12} dt \\
  &\les \int_{\mathbb{R}} \left(\int_{\mathbb{R}^3} |\phi|^{p+1}v_+^{\ga_0-1-\ep}(1+t)^{-\ga_0+1+\ep-(p-1)\a_p \ga_0}dx\right)^{\f12} dt\\
  &\les \left(\int_{\mathbb{R}}\int_{\mathbb{R}^3} |\phi|^{p+1}v_+^{\ga_0-1-\ep}dx dt\right)^{\f12}\left(\int (1+|t|)^{-\ga_0+1+\ep-(p-1)\a_p \ga_0}dt\right)^{\f12}.
\end{align*}
In view of the uniform spacetime bound of Proposition \ref{prop:spacetime:bd}, $\|\phi\|_{L_t^p L_x^{2p}}$ is finite if
\begin{align*}
  \ga_0-1+(p-1)\a_p \ga_0>1
\end{align*}
by choosing $\ep$ sufficiently small. As $1<\ga_0<p-1$, by choosing $\ga_0$ sufficiently close to $p-1$, it is equivalent to that
\begin{align*}
  f(p)=p-2+(p-1)^2 \frac{3+(p-2)^2}{(5-p)(p+1)}-1>0.
\end{align*}
It can be checked that there is a unique solution $p_*$ of $f(p)$ on $[2, 3]$ and when $p>p_*$, one has $f(p)>0$. Numerically, one can show that
\[
2.3541<p_*<2.3542.
\]
This proves the scattering result of Corollary \ref{cor:scattering:3D}.

\section{Proof for the uniform boundedness of the solution}
In this section, we proof Theorem \ref{thm:main:p2} for the case when $\frac{3}{2}<p\leq 2$. The method used to study the asymptotic decay properties for the solutions with power $p>2$ fails when $p\leq 2$, mainly due to the reason that we are not allowed to use weighted vector fields $X^{\gamma}$ with $\ga<1$. However, we still can use the vector fields $r^{\gamma}L$ as long as $\gamma\leq p-1$. To show the uniform boundedness of the solution, we begin with a weighted energy estimates through backward light cone obtained by using the vector fields $r^\gamma (\pa_t+\pa_r)$ as multipliers. 
\begin{Prop}
\label{prop:EF:cone:NW:3d:rW}
Assume that $1<p<5$.
Let $q=(t_0, x_0)$ be any point in $\mathbb{R}^{1+3}$. Then for solution $\phi$ of the nonlinear wave equation \eqref{eq:NLW:semi:3d} and for all $0\leq \ga\leq  \min\{ p-1, 2\}$, we have the following uniform bound
\begin{equation}
\label{eq:Eflux:ex:EF:rW}
\begin{split}
&\int_{\mathcal{N}^{-}(q)}((1+\tau)r^{\ga}+1)|\phi|^{p+1} d\si 
\leq C \mathcal{E}_{0, \ga}
\end{split}
\end{equation}
for some constant $C$ depending only on $p$. 
\end{Prop}
\begin{proof}
The proof goes similar to that for Proposition \ref{prop:EF:cone:NW:3d} by using the vector field $r^{\ga}L$ as multipliers instead of $X^{\ga}$. The use of the vector field $r^\ga L$ plays a crucial role in \cite{yang:scattering:NLW} for the potential energy decay (see Proposition \ref{prop:spacetime:bd}). Compared to that in \cite{yang:scattering:NLW}, we apply the vector field to  the region $\mathcal{J}^{-}(q)$ instead of regions bounded by null hypersurfaces. Hence it suffices to compute the boundary terms on $\mathcal{N}^{-}(q)$ and we will recall those necessary computations in \cite{yang:scattering:NLW}.

In the energy identity \eqref{eq:energy:id}, choose the vector fields $X$, $Y$ and the function $\chi$ as follows:
\[
X=r^{\gamma} L,\quad Y=\frac{1}{2}\ga r^{\ga-2}|\phi|^2 L, \quad \chi=r^{\ga-1},\quad \forall 0\leq \ga\leq \min\{p-1, 2\}.
\]
For the bulk integral on the right hand side of \eqref{eq:energy:id}, direct computation (or see \cite{yang:scattering:NLW}) shows that
\begin{align*}
&div(Y)+T[\phi]^{\mu\nu}\pi^X_{\mu\nu}+
\chi \pa_\mu\phi \pa^\mu\phi  +\chi\phi\Box\phi-\f12 \Box\chi |\phi|^2\\
&=\f12 r^{\ga-3}(\ga|L(r\phi)|^2+(2-\ga)|\nabb(r\phi)|^2 )+\frac{p-1-\ga }{p+1} r^{\ga-1}|\phi|^{p+1} .
\end{align*}
For the case when $0\leq  \ga \leq \min\{p-1, 2\}$, this term is nonnegative.

Let the domain $\mathcal{D}$ be $\mathcal{J}^{-}(q)$ with boundary $\B_{(0, x_0)}(t_0)\cup \mathcal{N}^{-}(q)$. By using Stokes' formula, the left hand side of the above energy identity consists of the integrals on the initial hypersurface $\B_{(0, x_0)}(t_0)$ and the backward light cone $\mathcal{N}^{-}(q)$.
For the integral on $\B_{(0, x_0)}(t_0)$, recall from \cite{yang:scattering:NLW} that
\begin{align}
\label{eq:PWE:ex:bxt01}
 \int_{\B_{(0, x_0)}(t_0)} i_{J^{X, Y,\chi}[\phi]}d\vol 
&=\f12 \int_{\B_{(0, x_0)}(t_0)}  r^\ga(  r^{-2}|L(r\phi)|^2 +|\nabb\phi|^2+\frac{2|\phi|^{p+1}}{p+1})  -\pa_r  (  r^{1+\ga} |\phi|^2 )r^{-2}dx.
\end{align}
 For the boundary integral on the backward light cone $\mathcal{N}^{-}(q)$, we compute the explicit form under the coordinates centered at the point $q=(t_0, x_0)$.
 Recall the volume form
\[
d\vol=dxdt=d\tilde{x}d\tilde{t}=2\tilde{r}^2 d\tilde{v}d\tilde{u}d\tilde{\om}.
\]
Under these new coordinates $(\tilde{t}, \tilde{x})$, we can compute that
\begin{align*}
-i_{J^{X, Y,\chi}[\phi]}d\vol=J_{\tilde{\Lb}}^{X, Y,\chi}[\phi]\tilde{r}^2d\tilde{u}d\tilde{\om}= ( T[\phi]_{\tilde{\Lb}\nu}X^\nu -
\f12(\tilde{\Lb}\chi) |\phi|^2 + \f12 \chi\cdot\tilde{\Lb}|\phi|^2 +Y_{\tilde{\Lb}}) \tilde{r}^2d\tilde{u}d\tilde{\om}.
\end{align*}
For the main quadratic terms, we have
\begin{align*}
T[\phi]_{\tilde{\Lb}\nu}X^\nu =T[\phi]_{\tilde{\Lb}\tilde{\Lb}}X^{\tilde{\Lb}}+T[\phi]_{\tilde{\Lb}\tilde{L}}X^{\tilde{L}}+T[\phi]_{\tilde{\Lb}\tilde{e}_i}X^{\tilde{e}_i}.
\end{align*}
Now we need to write the vector field $X$ under the new null frame $\{\tilde{L}, \tilde{\Lb}, \tilde{e}_1, \tilde{e}_2\}$ centered at $q$. Note that
\begin{align*}
\pa_r=\om \cdot \nabla=\om \cdot \tilde{\nabla}=\om\cdot \tilde{\om}\pa_{\tilde{r}}+ \om\cdot (\tilde{\nabla}-\tilde{\om}\pa_{\tilde{r}}).
\end{align*}
Hence we can write that
\begin{align*}
X=r^\gamma (\pa_t+\pa_r)
&=r^\ga (\pa_{\tilde{t}}+\om\cdot \tilde{\om}\pa_{\tilde{r}}+ \om\cdot \tilde{\nabb})=\f12 r^\ga(1+\om\cdot \tilde{\om})\tilde{L}+\f12 r^\ga(1-\om\cdot \tilde{\om}) \tilde{\Lb}+r^\ga \om\cdot \tilde{\nabb}.
\end{align*}
Here $\tilde{\nabb}=\tilde{\nabla}-\tilde{\om}\pa_{\tilde{r}}$. Denote $\tau=\om\cdot \tilde{\om}$. Then we can compute the quadratic terms
\begin{align*}
T[\phi]_{\tilde{\Lb}\nu}X^\nu 
=& \f12 (1-\tau)r^\ga |{\tilde{\Lb}}\phi|^2 + \f12 (1+\tau)r^\ga (|\tilde{\nabb}\phi|^2+\frac{2}{p+1}|\phi|^{p+1})+r^\ga  ({\tilde{\Lb}}\phi) (\om\cdot \tilde{\nabb})\phi.
\end{align*}
These terms are nonnegative. Indeed note that
\begin{align*}
\tilde{\Lb}(r)=-\tilde{\om}_i\pa_i(r)=-\tilde{\om}\cdot \om =-\tau,\quad \tilde{\nabb}(r)=(\tilde{\nabla}-\tilde{\om}\pa_{\tilde{r}})(r)=\om-\tilde{\om}\tau.
\end{align*}
Therefore we can write
\begin{align*}
&-\f12 r^2 (\tilde{\Lb}{\chi})|\phi|^2+\f12 r^2\chi \tilde{\Lb}|\phi|^2+r^2 Y_{\tilde{\Lb}}=r^{\ga}( {\tilde{\Lb}}(r\phi)+\tau \phi) \phi+\f12\tau (\ga-1)r^{\ga}|\phi|^2-\f12\ga (1+\tau)r^{\ga}|\phi|^2,\\
&r^2|\tilde{\Lb}\phi|^2=|{\tilde{\Lb}}(r\phi)-\tilde{\Lb}(r)\phi|^2=|{\tilde{\Lb}}(r\phi)|^2+\tau^2|\phi|^2+2 {\tilde{\Lb}}(r\phi) \tau\phi,\\
& r^2|\tilde{\nabb}\phi|^2
=|\tilde{\nabb}(r\phi)|^2+(1-\tau^2)|\phi|^2-2(\om-\tilde{\om}\tau)\cdot\tilde{\nabb}(r\phi) \phi,\\
& r^2 ({\tilde{\Lb}}\phi)  (\om\cdot \tilde{\nabb})\phi={\tilde{\Lb}}(r\phi) (\om \cdot \tilde{\nabb})(r\phi)-\tau(1-\tau^2)|\phi|^2+\phi \tau(\om\cdot \tilde{\nabb})(r\phi) -(1-\tau^2){\tilde{\Lb}}(r\phi)\phi.
\end{align*}
Notice that
\begin{align*}
  |(\om\cdot \tilde{\nabb})(r\phi)|=|(\om\times \tilde{\om}\cdot \tilde{\nabb})(r\phi)|=\sqrt{1-\tau^2}|\tilde{\nabb}(r\phi)|.
\end{align*}
In particular the quadratic terms are nonnegative
\begin{align*}
&\f12 (1-\tau)r^\ga |\tilde{\Lb}(r\phi)|^2+\f12(1+\tau)r^\ga |\tilde{\nabb}(r\phi)|^2+r^\ga {\tilde{\Lb}}(r\phi) (\om \cdot \tilde{\nabb})(r\phi)\geq 0.
\end{align*}
For the lower order terms, we compute that
\begin{align*}
&\f12 r^{\ga}(1-\tau)(\tau^2|\phi|^2+2{\tilde{\Lb}}(r\phi)\tau\phi )+r^{\ga}( {\tilde{\Lb}}(r\phi)+\tau \phi) \phi+\f12\tau (\ga-1)r^{\ga}|\phi|^2-\f12\ga (1+\tau)r^{\ga}|\phi|^2\\
&+\f12 r^{\ga}(1+\tau) \big((1-\tau^2)|\phi|^2-2(\om-\tilde{\om}\tau)\tilde{\nabb}(r\phi) \phi\big)\\
&+r^{\ga}\big(-\tau(1-\tau^2)|\phi|^2+\phi \tau(\om\cdot \tilde{\nabb})(r\phi)-\phi (1-\tau^2){\tilde{\Lb}}(r\phi) \big)\\
&=\f12(1-\ga)r^{\ga}|\phi|^2
+r^{\ga}(\tau {\tilde{\Lb}}-\om\cdot \tilde{\nabb})(r\phi) \phi\\
&=-\f12 r^2\tilde{r}^{-1} \tilde{\Om}_{ij}(r^{\ga-1} \om_j\tilde{\om}_i |\phi|^2)+\f12 \tilde{r}^{-2}r^2\tilde{\Lb}(r^{\ga-1}\tau \tilde{r}^2 |\phi|^2)\\
&+\f12(1-\ga)r^\ga|\phi|^2-\f12 \tilde{r}^{-2}r^2\tilde{\Lb}(r^{\ga-3}\tau\tilde{r}^2) |r\phi|^2+\f12 r^2 \tilde{r}^{-1}\tilde{\Om}_{ij}(r^{\ga-3} \om_j\tilde{\om}_i) |r\phi|^2.
\end{align*}
Here $\tilde{\Om}_{ij}=\tilde{x}_i\tilde{\pa}_{j}-\tilde{x}_j\tilde{\pa}_{{i}}=\tilde{r}(\tilde{\om}_i\tilde{\pa}_j-\tilde{\om}_j\tilde{\pa}_i)$ and we have omitted the summation sign for repeated indices $i$, $j$ for simplicity. By computations, note that
\begin{align*}
&\tilde{r}^{-1}\tilde{\Om}_{ij}(r^{-3}\om_j\tilde{\om}_i)=-2r^{\ga-4}(1-2\tau^2)-2\tau \tilde{r}^{-1}r^{\ga-3}+\ga r^{\ga-4}(1-\tau^2),\\
&\tilde{r}^{-2}r^{4-\ga}\tilde{\Lb}(r^{\ga-3}\tilde{r}^2\tau)=4\tau^2-1-2r\tilde{r}^{-1}\tau-\ga\tau^2.
\end{align*}
Thus the last line in the previous equality vanishes
\begin{align*}
  &\f12(1-\ga)r^{\ga}|\phi|^2-\f12 \tilde{r}^{-2}r^2\tilde{\Lb}(r^{\ga-3}\tau\tilde{r}^2) |r\phi|^2+\f12 r^2 \tilde{r}^{-1}\tilde{\Om}_{ij}(r^{\ga-3} \om_j\tilde{\om}_i) |r\phi|^2\\
  &=\f12 r^{\ga}|\phi|^2\left(1-\ga-4\tau^2+1+2r\tilde{r}^{-1}\tau+\ga\tau^2-2 (1-2\tau^2)-2\tau \tilde{r}^{-1}r +\ga  (1-\tau^2) \right)\\
  &=0.
\end{align*}
By using integration by parts on the backward light cone $\mathcal{N}^{-}(q)$, the integral of the second last line is
\begin{align*}
  &\int_{\mathcal{N}^{-}(q)}\big(-\f12 r^2\tilde{r}^{-1} \tilde{\Om}_{ij}(r^{\ga-3} \om_j\tilde{\om}_i |r\phi|^2)+\f12 \tilde{r}^{-2}r^2\tilde{\Lb}(  r^{\ga-1}\tau\tilde{r}^2 |\phi|^2)\big)r^{-2}\tilde{r}^2 d\tilde{u}d\tilde{\om}\\
  &= \f12 \int_{\S_{(0, x_0)}(t_0)}r^{\ga-1}\tau \tilde{r}^2 |\phi|^2d\tilde{\om}.
\end{align*}
The above computations show that the quadratic terms are nonnegative and the lower order terms are equal to the above integral on the 2-sphere on the initial hypersurface, that is,
\begin{equation*}
\begin{split}
&-\int_{\mathcal{N}^{-}(q)}i_{J^{X, Y, \chi}[\phi]}d\vol+ \f12 \int_{\S_{(0, x_0)}(t_0) }r^{\ga-1}\tau \tilde{r}^2 |\phi|^2d\tilde{\om}\geq \int_{\mathcal{N}^{-}(q)} \frac{1+\tau}{p+1} r^{\ga}|\phi|^{p+1}  \tilde{r}^2 d\tilde{u}d\tilde{\om}.
\end{split}
\end{equation*}
On the other hand for the case when $0\leq \ga \leq \min\{p-1, 2\}$, the bulk integral on the right hand side of the energy identity \eqref{eq:energy:id} is nonnegative, that is,
\begin{align*}
  \int_{\mathcal{N}^{-}(q)}i_{J^{X, Y, \chi}[\phi]}d\vol+\int_{\B_{(0, x_0)}(t_0)}i_{J^{X, Y, \chi}[\phi]}d\vol\geq 0.
\end{align*}
Adding this estimate to the previous inequality and in view of the expression \eqref{eq:PWE:ex:bxt01}, we conclude that
\begin{align*}
   &\int_{\B_{(0, x_0)}(t_0)}  r^\ga(  r^{-2}|L(r\phi)|^2 +|\nabb\phi|^2+\frac{2}{p+1}|\phi|^{p+1})  -\pa_r  (  r^{1+\ga} |\phi|^2 )r^{-2}dx\\
   &+  \int_{\S_{(0, x_0)}(t_0) }r^{\ga-1}\tau \tilde{r}^2 |\phi|^2d\tilde{\om}\geq 2\int_{\mathcal{N}^{-}(q)} \frac{1+\tau}{p+1} r^{\ga}|\phi|^{p+1}  \tilde{r}^2 d\tilde{u}d\tilde{\om}.
\end{align*}
Note that
\begin{align*}
  &\int_{\B_{(0, x_0)}(t_0)}   \pa_r  (  r^{1+\ga} |\phi|^2 )r^{-2}dx=\int_{\B_{(0, x_0)}(t_0)}  \textnormal{div} (\om r^{\ga-1}|\phi|^2)dx
  =\int_{\S_{(0, x_0)}(t_0) }r^{\ga-1}\tau \tilde{r}^2 |\phi|^2d\tilde{\om}.
\end{align*}
 We therefore derive from the previous inequality that
 \begin{align*}
   2\int_{\mathcal{N}^{-}(q)} \frac{1+\tau}{p+1} r^{\ga}|\phi|^{p+1}  d\sigma\leq  \int_{\B_{(0, x_0)}(t_0)}  r^\ga(  r^{-2}|L(r\phi)|^2 +|\nabb\phi|^2+\frac{2}{p+1}|\phi|^{p+1})  dx \leq \mathcal{E}_{0, \ga}.
 \end{align*}
The uniform bound \eqref{eq:Eflux:ex:EF:rW} of the Proposition then follows by the standard energy estimates obtained by using the vector field $\pa_t$ as multiplier.
\end{proof}

We now make use of the above weighted energy flux bound to prove the uniform boundedness of the solution for $\frac{3}{2}<p\leq 2$.
This relies on the following integration lemma.
 
\begin{Lem}
\label{lem:integration:smallp}
  Assume $1<p\leq 2$ and $0\leq \ga< p-1$. Fix $q=(t_0, x_0)$ in $\mathbb{R}^{1+3}$. For the $2$-sphere $\S_{(t_0-\tilde{r}, x_0)}(\tilde{r})$ on the backward light cone $\mathcal{N}^{-}(q)$, we have
  \begin{equation*}
    \begin{split}
    &\int_{\S_{(t_0-\tilde{r}, x_0)}(\tilde{r})} ((1+\tau)r^{\ga}+1)^{-p} d\tilde{\om} \leq C (1+r_0+\tilde{r})^{-\ga}
    \end{split}
  \end{equation*}
  for some constant $C$ depending only on $p$ and $\ga$.
  Here $\tau=\om\cdot \tilde{\om}$, $r_0=|x_0|$ and $0\leq \tilde{r}\leq t_0 $.
\end{Lem}
\begin{proof}
During the proof the implicit constant in $\les$ may also rely on $\ga$.
Denote $s=-\om_0\cdot \tilde{\om}$ and $\om_0=r_0^{-1}x_0$. By definition, recall that
  \begin{align*}
    &r^2=|x_0+\tilde{x}|^2=\tilde{r}^2+r_0^2-2r_0\tilde{r}s=(\tilde{r}-r_0s)^2+(1-s^2)r_0^2,\\
    &(1+\tau)r=r+r\om\cdot \tilde{\om}=r+(\tilde{x}+x_0)\cdot \tilde{\om}=r+\tilde{r}-r_0s.
  \end{align*}
We can write the integral as
  \begin{align*}
    &\int_{\S_{(t_0-\tilde{r}, x_0)}(\tilde{r})} ((1+\tau)r^{\ga}+1)^{-p} d\tilde{\om} =2\pi\int_{-1}^1 (r^{\ga-1}(r+\tilde{r}-r_0s)+1)^{-p}ds.
  \end{align*}
First we consider the case when the point $q$ locates in the exterior region, that is, $t_0\leq r_0$. The case when $t_0+r_0\leq 10$ is trivial. Thus in the following we always assume that $t_0+r_0\geq 10$. In particular in the exterior region, $r_0\geq 5$.
For the integral on $s\leq 0$, we trivially bound that
  \begin{align*}
    r^{\ga-1}(r+\tilde{r}-r_0s)\geq r_0^{\ga}.
  \end{align*}
  Therefore we can estimate that
  \begin{align*}
    \int_{-1}^0  (r^{\ga-1}(r+\tilde{r}-r_0s)+1)^{-p}ds\leq r_0^{-p\ga}.
  \end{align*}
  Define $s_0=1-(1-\tilde{r}r_0^{-1})^2$. On the interval $[0, s_0]$, note that
  \begin{align*}
    \sqrt{1-s} \ r_0\geq r_0-\tilde{r}.
  \end{align*}
Since $\tilde{r}\leq t_0\leq r_0$ and $0\leq s\leq s_0\leq 1$, we have
  \begin{align*}
    \tilde{r}-r_0s \leq r_0(1-s)\leq r_0\sqrt{1-s},\quad r_0s-\tilde{r}\leq r_0-\tilde{r}\leq r_0\sqrt{1-s}.
  \end{align*}
Define the relation $\sim $ meaning that two quantities are of the same size up to some universal constant, that is, $A\sim B$ means $C^{-1}B\leq A\leq CB$ for some constant $C$. The above computation shows that for $0\leq s\leq s_0$
\begin{align*}
  r\sim |\tilde{r}-r_0s|+\sqrt{1-s^2}r_0\sim \sqrt{1-s}r_0.
\end{align*}
Moreover when $\tilde{r}\geq r_0s$, it trivially has
\begin{align*}
  r+\tilde{r}-r_0s\sim \sqrt{1-s}r_0.
\end{align*}
Otherwise when $\tilde{r}<r_0 s$, note that 
\begin{align*}
  r+\tilde{r}-r_0s=\frac{r^2-(\tilde{r}-r_0s)^2}{r+r_0s-\tilde{r}}=\frac{(1-s^2)r_0^2}{r+r_0s-\tilde{r}}\sim\frac{(1-s^2)r_0^2}{\sqrt{1-s}r_0}\sim \sqrt{1-s}r_0.
\end{align*}
  Therefore on the interval $[0, s_0]$, we can estimate that
  \begin{align*}
    \int_{0}^{s_0}  (r^{\ga-1}(r+\tilde{r}-r_0s)+1)^{-p} ds
    &\les  \int_{0}^{s_0}(\sqrt{1-s} r_0)^{-p\ga} ds \les  r_0^{-p\ga}.
  \end{align*}
  Here we may note that $p\ga < p(p-1)\leq 2$.

  Finally on the interval $[s_0, 1]$, notice that
  \begin{align*}
    r_0s\geq r_0s_0\geq \tilde{r},\quad \sqrt{1-s}r_0\leq r_0-\tilde{r}.
  \end{align*}
  Therefore we have
  \begin{align*}
    r\sim r_0s-\tilde{r}+\sqrt{1-s}r_0=r_0-\tilde{r}+(\sqrt{1-s}-(1-s))r_0\sim r_0-\tilde{r}.
  \end{align*}
  Hence we can estimate that
  \begin{align*}
   r(1+\tau)= r+\tilde{r}-r_0s=\frac{(1-s^2)r_0^2}{r+r_0s-\tilde{r}}\sim \frac{(1-s)r_0^2}{r_0-\tilde{r}}.
  \end{align*}
This leads to the bound 
  \begin{align*}
    &\int_{s_0}^{1}  (r^{\ga-1}(r+\tilde{r}-r_0s)+1)^{-p} ds\\
    &\les  \int_{s_0}^{1}\big(1+ (r_0-\tilde{r})^{\ga-2}(1-s)r_0^2\big)^{-p} ds\\
    &\les (r_0-\tilde{r})^{2-\ga} r_0^{-2} .
  \end{align*}
  Combining the above estimates, we have shown that in the exterior region $\{t_0\leq |x_0|\}$
  \begin{align*}
    \int_{\S_{(t_0-\tilde{r}, x_0)}(\tilde{r})} ((1+\tau)r^{\ga}+1)^{-p} d\tilde{\om} \les r_0^{-p\ga}+(r_0-\tilde{r})^{2-\ga} r_0^{-2}\les (1+r_0+\tilde{r})^{-\ga}.
  \end{align*}
This means that the Lemma holds for the case when $t_0\leq r_0$.

\bigskip

Next we consider the situation in the interior region when $t_0>r_0$ and $t_0+r_0>10$. The integral on $[-1, 0]$ is easy to control. Indeed, when $s\leq 0$, by the expression of $r$ and $\tau$, we have
\begin{align*}
  r\sim \tilde{r}+r_0,\quad r(1+\tau)\sim \tilde{r}+r_0.
\end{align*}
Therefore, we can show that
\begin{align*}
  \int_{-1}^{0}(1+r^{\ga}(1+\tau))^{-p}ds\les \int_{-1}^{0}(1+(\tilde{r}+r_0)^{\ga})^{-p}ds\les (1+\tilde{r}+r_0)^{-p\ga}.
\end{align*}
For the integral on the interval $[0, 1]$, when $\tilde{r}\geq 2r_0$, we have
\begin{align*}
  r\sim \tilde{r}-r_0s+\sqrt{1-s}r_0\sim \tilde{r},\quad
  r+\tilde{r}-r_0s\sim \tilde{r}.
\end{align*}
Since $r_0\leq \f12 \tilde{r}$, we can show that
\begin{align*}
  \int_0^{1}(1+r^{\ga}(1+\tau))^{-p}ds\les \int_0^{1}(1+\tilde{r}^{\ga})^{-p}ds\les (1+\tilde{r}+r_0)^{-p\ga}.
\end{align*}
Now for the case when $r_0\leq \tilde{r}\leq 2 r_0$, similarly $r$ and $r(1+\tau)$ behave like
\begin{align*}
  &r\sim \tilde{r}-r_0s+\sqrt{1-s}r_0\sim \tilde{r}-r_0+\sqrt{1-s}r_0,\\
  & r\leq r(1+\tau)=r+\tilde{r}-r_0s\leq 2r.
\end{align*}
This shows that
\begin{align*}
  \int_0^1 (1+r^{\ga}(1+\tau))^{-p}ds\les \int_0^{1}(1+\sqrt{1-s}r_0)^{-p\ga}ds\les (1+\tilde{r}+r_0)^{-p\ga}.
\end{align*}
Again here we used the assumption that $p\ga<2$ and $\tilde{r}\leq 2r_0$.

Finally when $0\leq \tilde{r}\leq r_0$, we split the integral on $[0, 1]$ into several parts. On $[0, r_0^{-1}\tilde{r}]$, similar to the above case when $r_0\leq \tilde{r}\leq 2r_0$, we can estimate that
\begin{align*}
  \int_0^{r_0^{-1}\tilde{r}} (1+r^{\ga}(1+\tau)^{\ga})^{-p}ds\les \int_0^{r_0^{-1}\tilde{r}}(1+\tilde{r}-r_0+\sqrt{1-s}r_0)^{-p\ga}ds\les (1+\tilde{r}+r_0)^{-p\ga}.
\end{align*}
On $[r_0^{-1}\tilde{r}, 1]$, firstly we have
\begin{align*}
  r&\sim r_0s-\tilde{r}+\sqrt{1-s}r_0,\\
   r(1+\tau)&=\frac{(1-s^2)r_0^2}{r+r_0s-\tilde{r}} \sim\frac{(1-s) r_0^2}{r_0s-\tilde{r}+\sqrt{1-s}r_0}.
\end{align*}
Therefore, we can bound that
\begin{align*}
  \int_{r_0^{-1}\tilde{r}}^{1} (1+r^{\ga}(1+\tau))^{-p}ds \les \int_{r_0^{-1}\tilde{r}}^{1}(1+(r_0s-\tilde{r}+\sqrt{1-s}r_0)^{\ga-2}(1-s)r_0^2)^{-p}ds.
\end{align*}
Notice that
\begin{align*}
\frac{1}{2}(r_0-\tilde{r}+\sqrt{1-s} r_0)\leq r_0s-\tilde{r}+\sqrt{1-s}r_0\leq r_0-\tilde{r}+\sqrt{1-s}r_0,\quad r_0^{-1}\tilde{r}\leq s\leq 1.
\end{align*}
Denote $s_*=1-(1-r_0^{-1}\tilde{r})^2$. In particular, we have
\begin{align*}
  r_0^{-1}\tilde{r}\leq s_*\leq 1.
\end{align*}
On the interval $[s_*, 1]$, we have
\begin{align*}
  \sqrt{1-s}r_0\leq r_0-\tilde{r}.
\end{align*}
Therefore we show that
\begin{align*}
  \int_{s_*}^{1}(1+(1+\tau)r^{\ga})^{-p}ds &\les \int_{s_*}^1(1+(r_0-\tilde{r})^{\ga-2}(1-s)r_0^2)^{-p}ds\\
  &\les (r_0-\tilde{r})^{2-\ga}r_0^{-2}(1-(1+(r_0-\tilde{r})^{\ga})^{1-p})\\
  &\les (1+r_0)^{-\ga}\\
  &\les (1+r_0+\tilde{r})^{-\ga}.
\end{align*}
Otherwise on $[r_0^{-1}\tilde{r}, s_*]$, we can bound that
\begin{align*}
  \int_{r_0^{-1}\tilde{r}}^{s_*}(1+(1+\tau)r^{\ga})^{-p}ds &\les \int^{s_*}_{r_0^{-1}\tilde{r}}(1+(1-s)^{\f12\ga}r_0^{\ga})^{-p}ds\\
  &\les \int^{s_*}_{r_0^{-1}\tilde{r}}(1+(1-s) r_0^{2})^{-\f12 p\ga}ds \\
  &\les r_0^{-2}(1+r_0(r_0-\tilde{r}))^{1-\f12 p\ga }\\
  &\les  (1+\tilde{r}+r_0)^{-p\ga}.
\end{align*}
Here keep in mind that $\tilde{r}\leq r_0$.
The Lemma holds by combining all the above bounds.
\end{proof}
 We now use the above integration lemma as well as the weighted energy estimate of Proposition \ref{prop:EF:cone:NW:3d:rW} to study the asymptotic behavior of the solution in the pointwise sense. 
 By using the standard Sobolev embedding and energy estimate, the linear evolution part in the representation formula \eqref{eq:rep4phi:ex} verifies
\begin{align*}
  |\int_{\tilde{\om}}t_0  \phi_1(x_0+t_0\tilde{\om})d\tilde{\om}+\pa_{t_0}\big(\int_{\tilde{\om}}t_0  \phi_0(x_0+t_0\tilde{\om})d\tilde{\om}   \big)|
  &\les \sqrt{\mathcal{E}_{1, 0}  }.
\end{align*}
To control the nonlinearity, in view of Lemma \ref{lem:integration:smallp} and Proposition \ref{prop:EF:cone:NW:3d:rW}, we can bound that
\begin{align*}
  &|\int_{\mathcal{N}^{-}(q) }|\phi|^{p-1}\phi \ \tilde{r} d\tilde{r}d\tilde{\om}|\\
  &\les \left(\int_{\mathcal{N}^{-}(q)}((1+\tau)r^{\ga}+1)|\phi|^{p+1}\ \tilde{r}^{2} d\tilde{r}d\tilde{\om}\right)^{\frac{p}{p+1}}  \cdot \left(\int_{\mathcal{N}^{-}(q) }((1+\tau)r^{\ga}+1)^{-p} \tilde{r}^{1-p} d\tilde{r}d\tilde{\om}\right)^{\frac{1}{p+1}}\\
  &\les \mathcal{E}_{0, \ga}^{\frac{p}{p+1}} \left(\int_{0}^{t_0}  (1+\tilde{r})^{-\ga} \tilde{r}^{1-p} d\tilde{r} \right)^{\frac{1}{p+1}}\\
  &\les \mathcal{E}_{0, p-1}^{\frac{p}{p+1}} \left(1+(1+t_0)^{2-p-\ga}\right)^{\frac{1}{p+1}}\\
  &\les \mathcal{E}_{0, p-1}^{\frac{p}{p+1}}
\end{align*}
by choosing $\ga=2-p+\ep$ such that $0<\ep<2p-3$. This shows that
\begin{align*}
  |\phi|\les \sqrt{\mathcal{E}_{1, 0}} +\mathcal{E}_{0, p-1}^{\frac{p}{p+1}}.
\end{align*}
By our definition, the implicit constant relies only on $p$. Hence the uniform boundedness of the main Theorem \ref{thm:main:p2} follows.

\bibliography{shiwu}{}

\begin{thebibliography}{10}

\bibitem{Bieli:3DNLW}
R.~Bieli and N.~Szpak.
\newblock {Large data pointwise decay for defocusing semilinear wave
  equations}.
\newblock 2010.
\newblock ar{X}iv:1002.3623.

\bibitem{Brenner75:Lp:LW}
P.~Brenner.
\newblock On {$L_{p}-L_{p^{\prime} }$} estimates for the wave-equation.
\newblock {\em Math. Z.}, 145(3):251--254, 1975.

\bibitem{Brenner79:globalregularity:NW}
P.~Brenner.
\newblock On the existence of global smooth solutions of certain semilinear
  hyperbolic equations.
\newblock {\em Math. Z.}, 167(2):99--135, 1979.

\bibitem{Brenner81:globalregularity:d9}
P.~Brenner and W.~von Wahl.
\newblock Global classical solutions of nonlinear wave equations.
\newblock {\em Math. Z.}, 176(1):87--121, 1981.

\bibitem{ChristodoulouYangM}
Y.~Choquet-Bruhat and D.~Christodoulou.
\newblock Existence of global solutions of the {Y}ang-{M}ills, {H}iggs and
  spinor field equations in {$3+1$} dimensions.
\newblock {\em Ann. Sci. \'Ecole Norm. Sup. (4)}, 14(4):481--506 (1982), 1981.

\bibitem{ChDNull}
D.~Christodoulou.
\newblock Global solutions of nonlinear hyperbolic equations for small initial
  data.
\newblock {\em Comm. Pure Appl. Math.}, 39(2):267--282, 1986.

\bibitem{Dafermos17:C0Kerr}
M.~Dafermos and J.~Luk.
\newblock {The interior of dynamical vacuum black holes I:The $C^0$-stability
  of the Kerr Cauchy horizon}.
\newblock 2017.
\newblock ar{X}iv:1710.01722.

\bibitem{newapp}
M.~Dafermos and I.~Rodnianski.
\newblock A new physical-space approach to decay for the wave equation with
  applications to black hole spacetimes.
\newblock In {\em X{VI}th {I}nternational {C}ongress on {M}athematical
  {P}hysics}, pages 421--432. World Sci. Publ., Hackensack, NJ, 2010.

\bibitem{Moncrief1}
D.~Eardley and V.~Moncrief.
\newblock The global existence of {Y}ang-{M}ills-{H}iggs fields in
  {$4$}-dimensional {M}inkowski space. {I}. {L}ocal existence and smoothness
  properties.
\newblock {\em Comm. Math. Phys.}, 83(2):171--191, 1982.

\bibitem{Moncrief2}
D.~Eardley and V.~Moncrief.
\newblock The global existence of {Y}ang-{M}ills-{H}iggs fields in
  {$4$}-dimensional {M}inkowski space. {II}. {C}ompletion of proof.
\newblock {\em Comm. Math. Phys.}, 83(2):193--212, 1982.

\bibitem{velo85:global:sol:NLW}
J.~Ginibre and G.~Velo.
\newblock The global {C}auchy problem for the nonlinear {K}lein-{G}ordon
  equation.
\newblock {\em Math. Z.}, 189(4):487--505, 1985.

\bibitem{Velo87:decay:NLW}
J.~Ginibre and G.~Velo.
\newblock Conformal invariance and time decay for nonlinear wave equations.
  {I}, {II}.
\newblock {\em Ann. Inst. H. Poincar\'{e} Phys. Th\'{e}or.}, 47(3):221--261,
  263--276, 1987.

\bibitem{Velo89:globalsolution:NLW}
J.~Ginibre and G.~Velo.
\newblock The global {C}auchy problem for the nonlinear {K}lein-{G}ordon
  equation. {II}.
\newblock {\em Ann. Inst. H. Poincar\'{e} Anal. Non Lin\'{e}aire}, 6(1):15--35,
  1989.

\bibitem{John79:blowup:NLW:3d}
F.~John.
\newblock Blow-up of solutions of nonlinear wave equations in three space
  dimensions.
\newblock {\em Manuscripta Math.}, 28(1-3):235--268, 1979.

\bibitem{Jorgens61:energysub:NLW:lowerd}
K.~J\"{o}rgens.
\newblock Das {A}nfangswertproblem im {G}rossen f\"{u}r eine {K}lasse
  nichtlinearer {W}ellengleichungen.
\newblock {\em Math. Z.}, 77:295--308, 1961.

\bibitem{LindbladMKG}
H.~Lindblad and J.~Sterbenz.
\newblock Global stability for charged-scalar fields on {M}inkowski space.
\newblock {\em IMRP Int. Math. Res. Pap.}, pages Art. ID 52976, 109, 2006.

\bibitem{Pecher76:NLW:global}
H.~Pecher.
\newblock {$L^{p}$}-{A}bsch\"{a}tzungen und klassische {L}\"{o}sungen f\"{u}r
  nichtlineare {W}ellengleichungen. {I}.
\newblock {\em Math. Z.}, 150(2):159--183, 1976.

\bibitem{Pecher82:decay:3d}
H.~Pecher.
\newblock Decay of solutions of nonlinear wave equations in three space
  dimensions.
\newblock {\em J. Funct. Anal.}, 46(2):221--229, 1982.

\bibitem{Pecher88:scattering:sharpp:3D}
H.~Pecher.
\newblock Scattering for semilinear wave equations with small data in three
  space dimensions.
\newblock {\em Math. Z.}, 198(2):277--289, 1988.

\bibitem{segal63:semigroup}
I.~Segal.
\newblock Non-linear semi-groups.
\newblock {\em Ann. of Math. (2)}, 78:339--364, 1963.

\bibitem{Strauss:NLW:decay}
W.~Strauss.
\newblock Decay and asymptotics for {$ cmu=F(u)$}.
\newblock {\em J. Functional Analysis}, 2:409--457, 1968.

\bibitem{Strauss78:NLW}
W.~Strauss.
\newblock Nonlinear invariant wave equations.
\newblock In {\em Invariant wave equations ({P}roc. ``{E}ttore {M}ajorana''
  {I}nternat. {S}chool of {M}ath. {P}hys., {E}rice, 1977)}, volume~73 of {\em
  Lecture Notes in Phys.}, pages 197--249. Springer, Berlin-New York, 1978.

\bibitem{vonWahl72:decay:NLW:super}
W.~von Wahl.
\newblock Some decay-estimates for nonlinear wave equations.
\newblock {\em J. Functional Analysis}, 9:490--495, 1972.

\bibitem{Vonwahl75:NW}
W.~von Wahl.
\newblock \"{U}ber nichtlineare {W}ellengleichungen mit zeitabh\"{a}ngigem
  elliptischen {H}auptteil.
\newblock {\em Math. Z.}, 142:105--120, 1975.

\bibitem{YangW:NLW:3d:p1}
D.~Wei and S.~Yang.
\newblock {On the defocusing semilinear wave in three space dimension with
  small power}.
\newblock preprint.

\bibitem{yang:scattering:NLW}
S.~Yang.
\newblock {Global behaviors for defocusing semilinear wave equations}.
\newblock {\em Ann. Sci. \'{E}c. Norm. Sup\'{e}r. (4)}.
\newblock to appear.

\end{thebibliography}
\bibliographystyle{plain}

\bigskip

Beijing International Center for Mathematical Research, Peking University,
Beijing, China

\textsl{Email}: shiwuyang@bicmr.pku.edu.cn

\end{document}